\documentclass[10pt, a4paper]{article}

\usepackage[utf8]{inputenc}
\usepackage[T1]{fontenc}
\usepackage{amsthm}
\usepackage{array}
\usepackage{amsmath}
\usepackage{amsfonts}
\usepackage{amssymb}
\usepackage{caption}
\usepackage{float}
\usepackage{graphicx}
\usepackage{ifthen}
\usepackage{tikz}
\usepackage{tikz-cd}
\usepackage{subcaption}
\usepackage{longtable}
\usepackage{fullpage}
\usepackage{textcomp}
\usepackage{mathtools}
\usepackage[all]{xy}
\usepackage{makecell}
\usepackage[hidelinks]{hyperref}
\usepackage{setspace}
\usepackage[a4paper]{geometry}
\usepackage{color}
\usepackage{amsmath}
\usepackage{algorithm}
\usepackage{afterpage}
\usepackage[noend]{algpseudocode}
\usepackage{listings}
\usepackage{chngcntr}

\usetikzlibrary{patterns}
\usetikzlibrary{intersections}

\newtheorem{theorem}{Theorem}
\newtheorem{construction}[theorem]{Construction}

\newtheorem{lemma}[theorem]{Lemma}

\newtheorem{conjecture}[theorem]{Conjecture}
\newtheorem{definition}[theorem]{Definition}
\newtheorem{example}[theorem]{Example}

\newcommand \Lawrence [1] {{#1}}

\algnewcommand{\LeftComment}[1]{\Statex \(\triangleright\) #1}

\counterwithin{theorem}{section}
\counterwithin{equation}{section}
\counterwithin{figure}{section}
\counterwithin{table}{section}

\def \Spec {\operatorname{Spec} \; }

\def \PP {\mathbb{P}}

\def \NN {\mathbb{N}}

\def \CC {\mathbb{C}}
\def \AA {\mathbb{A}}
\def \RR {\mathbb{R}}
\def \ZZ {\mathbb{Z}}

\def \lb {[ \! [ }
\def \rb {] \! ]}

\begin{document}

\title{Explicit equations for mirror families to log Calabi-Yau surfaces}
\author{Lawrence Jack Barrott}
\date{}

\maketitle

\begin{abstract}

Mirror symmetry for del Pezzo surfaces was studied in~\cite{AurouxKatzarkovOrlov} where they suggested that the mirror should take the form of a Landau-Ginzburg model with a particular type of elliptic fibration. This argument came from symplectic considerations of the derived categories involved. This problem was then considered again but from an algebro-geometric perspective by Gross, Hacking and Keel in~\cite{MirrorSymmetryforLogCY1}. Their construction allows one to construct a formal mirror family to a pair $(S,D)$ where $S$ is a smooth rational projective surface and $D$ a certain type of Weil divisor supporting an ample or anti-ample class. In the case where the self intersection matrix for $D$ is not negative semi-definite it was shown in~\cite{MirrorSymmetryforLogCY1} that this family may be lifted to an algebraic family over an affine base.

In this paper we perform this construction for all smooth del Pezzo surfaces of degree at least two and obtain explicit equations for the mirror families and explain some of the motivation for their construction. In the end we will provide enumerative data supporting the claim that these are in fact the desired mirror families.

\end{abstract}

\section{Introduction}

Our starting materials will be a pair $(S,D)$ where $S$ is a smooth del Pezzo surface of degree at least two and $D \in \mid -K_S \mid$ is a cycle of rational curves. Such pairs were first studied in~\cite{RationalSurfacesWithAnAntiCanonicalCycle} as relating to certain cusp singularities and therefore are called \emph{Looijenga pairs}. When viewed another way they are \emph{log Calabi-Yau surfaces}. Our goal will be to explicitly construct the mirror to $(S,D)$ in a variety of cases using the Gross-Siebert program.

Mirror symmetry predicts, according to Givental's 1994 ICM lecture, that associated to a toric Fano variety $V$ there should exist a mirror Landau-Ginzburg model (LG model). This is a pair $(\check{V}, \check{W})$ where $\check{V}$ is a smooth variety and $\check{W}: \check{V} \rightarrow \CC$ the \emph{super-potential}, a regular function. This was extended in work of Cho and Oh, who in~\cite{FloerCohomologyAndDiscInstantonsOfLagrangianTorusFibersInFanoToricManifolds} realised that a choice of effective anti-canonical divisor played a role in determining the Landau-Ginzburg model. We will introduce the toric construction and then explain how to correct this to the non-toric case to produce an explicit description of the mirror family of LG models in terms of algebraic equations over a known base.

  The paper~\cite{MirrorSymmetryforLogCY1} constructs a formal smoothing of the $n$-vertex, a cycle of planes meeting along their axes, over $\Spec k \lb NE(S) \rb$. The key result for us of~\cite{MirrorSymmetryforLogCY1} is the following.

\begin {theorem}

  Let $(S,D)$ be a pair as above and suppose that the self intersection matrix of $D$ is not negative semi-definite. Then the above formal smoothing lifts to an algebraic family over the entirety of $\Spec k [ NE(S) ]$.
  
\end{theorem}

\begin {proof}

  This is Theorem 0.2 of~\cite{MirrorSymmetryforLogCY1}.
  
\end{proof}

The entire Gross-Siebert program is driven by tropicalising the SYZ conjecture. We briefly describe this philosophy and will explain throughout the body of the paper how it motivates each step of the construction. First let us recall the SYZ conjecture:

\begin{conjecture}[The SYZ conjecture]

Let $X, \check{X}$ be a mirror pair of \emph{Calabi-Yau varieties}. Then there is an affine manifold with singularities $B$ and maps $\phi: X \rightarrow B$, $\check{\phi}: \check{X} \rightarrow B$ which are dual special Lagrangian fibrations. This is shown schematically below

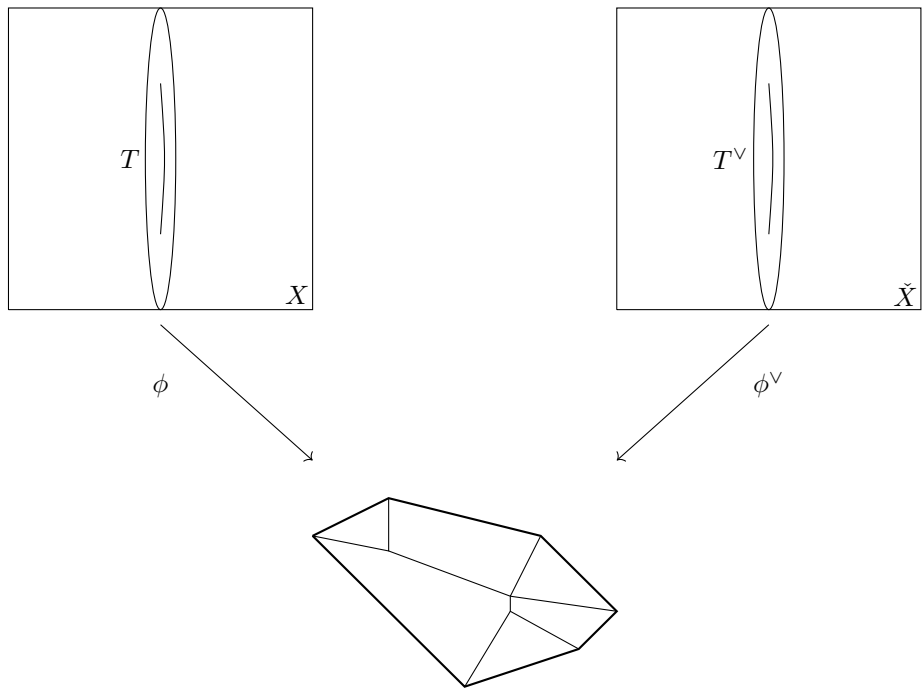
\begin{figure}[h]
\centering
\begin{tikzpicture}
\draw (-6,-2) -- (-2,-2) -- (-2,2) -- (-6,2) -- (-6,-2);
\draw (2,-2) -- (6,-2) -- (6,2) -- (2,2) -- (2,-2);
\draw (-2.2, -1.8) node {$X$};
\draw (5.8, -1.8) node {$\check{X}$};
\draw (-4,0) ellipse (0.2 and 2);
\draw (-4,1) .. controls (-3.93,0) .. (-4,-1);
\draw (4,0) ellipse (0.2 and 2);
\draw (4,1) .. controls (4.07,0) .. (4,-1);
\draw[->] (-4,-2.2) -- (-2,-4);
\draw[->] (4,-2.2) -- (2,-4);
\draw (-4.4, 0) node {$T$};
\draw (3.5, 0) node {$T^\vee$};
\draw (-4, -3) node {$\phi$};
\draw (4, -3) node {$\phi^\vee$};
\draw[thick] (-2,-5) -- (-1, -4.5) -- (1, -5) -- (2, -6) -- (1.5, -6.5) -- (0, -7) -- (-2,-5);
\draw (-2,-5) -- (-1, -5.2) -- (-1, -4.5);
\draw (-1, -5.2) -- (0.6, -5.8) -- (1,-5);
\draw (0.6, -5.8) -- (2,-6);
\draw (0.6, -5.8) -- (0.6,-6) -- (1.5,-6.5);
\draw (0.6, -6) -- (0,-7);
\end{tikzpicture}
\end{figure}

\end{conjecture}

There is a natural collection of varieties with such a special Lagrangian fibration, toric varieties with their momentum maps. The momentum map fibres a toric variety over a polytope for that variety and the fibres degenerate in a natural way over the boundary. These toric varieties form an important laboratory in which we can study mirror symmetry. Indeed Batyrev and Borisov used this structure to construct mirror Calabi-Yau pairs for ``good'' hypersurfaces in toric varieties. Here the duality is induced by replacing the polytope by the dual polytope.

The philosophy of Gross and Siebert is to assume the existence of such a fibration $X \rightarrow B$ to a known base supporting a choice of piecewise linear function $\phi$ taking values in $NE(S)$ and understand how singularities of the fibration affect the mirror $\check{X}$. For example the image of a curve under the fibration is a topological space, called a \emph{tropical amoeba}, retracting onto a balanced piecewise linear graph, as described in~\cite{AmoebasOfAlgebraicVarietiesAndTropicalGeometry}. Such a skeleton is called a \emph{tropical curve}, and the process of passing to this limit is called \emph{tropicalisation}. Gross and Siebert use these tropical curves to replace curves living on $X$, calling their analogue \emph{broken lines}.

An important symplectic invariant of a symplectic manifold $(S,\omega)$ is the symplectic cohomology. This is a ring whose objects are classes of Lagrangian submanifolds and whose multiplication counts certain Maslov index two disks filling the gap between specified Lagrangians. Now points of $B$ correspond to special Lagrangians and the image of one of these triangles is a trivalent graph on $B$, bending where Maslov index zero disks have been glued on. The expectation is that the underlying scheme of the mirror should be the spectrum of this ring. This fails to be true once one moves away from the dimension two case. There are examples due to Pomerleano where every fibre of the family is singular and one must resolve these singularities to obtain the correct mirror family. In general we expect that the family will be projective over an affine base.

In the toric case the choice of base $B$ embeds into $\RR^2$ and there are no index zero disks. Furthermore the base $B$ has one-strata $v_1 ,\ldots ,v_n$ corresponding to toric invariant divisors, $D_1, \ldots, D_n$. The construction becomes the following, for every point $P \in B_\ZZ$ introduce a variable $\vartheta_P$. They define a multiplication rule on the $\CC$ linear span of these by declaring that $\vartheta_P \vartheta_Q = z^{\phi (P+Q) - \phi(P) - \phi(Q)} \vartheta_{P+Q}$\label{ToricThetaProduct}. This precisely replicates the counts of disks mentioned above. With this product rule one produces an associative ring whose spectrum fibres over $\Spec k[NS(S)]$ and the mirror family is this family together with the superpotential $\sum \vartheta_{v_i}$. 

Now in the general case the presence of singularities of the SYZ fibration produces Maslov index zero disks in a controlled way. Gross, Hacking and Keel mimic this by introducing the \emph{canonical scattering diagram} on $B$, a combinatorial structure which records the gluing possibilities for Maslov index zero disks. The techniques themselves were introduced by Kontsevich and Soibelman in their paper~\cite{AffineStructuresAndNonArchimedeanAnalyticSpaces} to study $K3$ surfaces. These were applied by Gross in~\cite{TropicalGeometryAndMirrorSymmetry} to study the prototype example of the construction, mirror symmetry for $\PP^2$, and by Gross and Siebert in~\cite{AnInvitationToToricDegenerations} to construct mirror partners to Calabi-Yau threefolds. It was expanded upon in work of both~\cite{TheTropicalVertex} and~\cite{ATropicalViewOfLandauGinzburgModels} to construct mirrors to Fano varieties.

We begin by constructing the base of the SYZ fibration, as described in~\cite{MirrorSymmetryforLogCY1} and then perform combinatorics to limit the information we must calculate. The salient point for the non negative semi-definite case is that there is a positivity requirement that restricts the possible pairs of broken lines which may contribute in the product. One may think of this as analogous to the convergence of quantum cohomology for Fano manifolds. For high degree the calculations may be explicitly checked by hand, whilst for low degree we must rely on computer algebra packages. The techniques mentioned here can be applied to $dP_1$ though the explicit equations produced will potentially be too complicated to reproduce in a meaningful way. 

Having found such equations this we will provide numerological evidence that the constructed families are indeed the correct mirror family. One prediction from mirror symmetry would be an isomorphism between (an appropriate enhancement of) the Fukaya category of $V$ and the category of matrix factorisations of the potential $\check {W}$. To make this more precise one can pass to the Hochschild cohomology of these categories.

  Conjecturally, and with promising results by Bourgeois, Ekholm and Eliashberg in~\cite{EffectOfLegendrianSurgery}, the Hochschild cohomology of the Fukaya category of $S$ is canonically isomorphic to the quantum cohomology of $V$. Recall that the quantum cohomology is the deformation over $\CC [ H_2(S)]$ of the usual intersection form on the cohomology ring $H^*(V)$. The quantum product $H_1*H_2$ of classes $H_1$ and $H_2$ is defined by
 \[ \langle H_1 * H_2, H_3 \rangle = \sum_{\alpha \in H_2 (V)} I_{0,3}(V, \alpha, H_1, H_2, H_3) z^\alpha \]
 Where $I_{0,3} (X, \alpha, H_1, H_2, H_3)$ counts rational curves in class $\alpha$ with three marked points lying on \Lawrence{classes Poincar\'e dual} to $H_1$, $H_2$ and $H_3$. This product converges under some mild positivity constraints and remains associative with the central fibre reproducing the classical cup product. The details of this construction are described in~\cite{MirrorSymmetryAndAlgebraicGeometry} Theorem 8.1.4.

 On the mirror side the Hochschild cohomology of the category of matrix factorisations is isomorphic to the Jacobian ring $Jac (\tilde {W})$ of the critical locus of $\check{W}$, at least in the case where the critical locus is isolated. Recall that this is the ring
 \[ k[x_1, \ldots x_n] / \langle \partial \check{W} / \partial x_i \rangle\]
Since mirror symmetry is supposed to provide us with an isomorphism between the two $dg$-categories it should also provide an isomorphism between the corresponding Hochschild cohomologies. In the final section we will look for numerological evidence that this is indeed the case. A general proof of the existence of this isomorphism should follow from a deeper understanding of the gluing requirements for Gromov-Witten invariants.

\subsubsection {Acknowledgements}

This work was an initial project in my PhD thesis, as suggested by Mark Gross. It was from him that I learnt much of the background and motivation for the Gross-Siebert program. Various discussions about the material of~\cite{MirrorSymmetryforLogCY1} with Tyler Kelly, Zhi Jin, Ben Morely and Andrea Petracci proved immensely helpful to my understanding of the material. This project was included in my thesis and I must thank Tom Coates and Pelham Wilson for examining me, and ironing out various parts of the exposition and of course Mark Gross who helped to review the initial draft of this paper and suggesting the problem.

This project was funded by a variety of sources: by an Internal Graduate Studentship provided by Trinity College, by a research studentship provided by the Cambridge Philosophical Society and by Final Term Funding provided by the Department for Pure Mathematics and Mathematical Statistic (DPMMS) in Cambridge. I wrote this paper as a research assistant at the NCTS in Taipei.  

\section {Constructions of mirror families to del Pezzo surfaces}

Throughout this paper let $S$ be a smooth complex del Pezzo surface and $D = D_1 \cup D_2 \cup \ldots D_n$ an anti-canonical cycle of rational curves with $n$ at least three (we will explain later how to deal with the case of $dP_2$). Recall that by the classification of del Pezzo surfaces this ensures that $S$ is isomorphic to one of the following: 

\begin{itemize}
\item $\PP ^1 \times \PP^1$.
\item The blow-up of $\PP^2$ in zero through eight points in general position.
\end{itemize}

The respective Chow groups $A_1 (S)$ are generated by

\begin{itemize}
\item The pullback of a point under the two projection maps.
\item A hyperplane class $H$ and the exceptional curves $E_1, \ldots , E_{9-d}$.
\end{itemize}
whilst the groups $A_0(S)$ and $A_2 (S)$ are isomorphic to $\ZZ$ generated by a point and a fundamental class respectively.

This pair can be thought of in many ways, but perhaps the most appealing to those studying mirror symmetry is as a \emph{log Calabi-Yau varieties}, varieties such that there is a non-vanishing differential form on $S\setminus D$ with at worst logarithmic singularities along $D$. Of these the surfaces $\PP^1 \times \PP^1$, $\PP^2, dP_8, dP_7$ and $dP_6$ are all toric and mirror families can be constructed using techniques we mentioned in the introduction. The next three surfaces were studied in~\cite{MirrorSymmetryForDelPezzoSurfacesVanishingCyclesAndCoherentSheaves}. To build up our intuition with the Gross-Siebert program we will show that these agree with the predictions there. We will then handle one of the remaining two cases, the blow up in seven points. Precisely the same techniques can be applied in the case of $dP_1$ but given the frankly atrocious nature of the equations for $dP_2$ we do not pursue this.

In the toric case the base is a fan for the variety in the sense of~\cite{ToricVarieties}. Recall that a fan is a simplicial complex of cones in $\RR^n$, where a cone is the $\RR^{+}$-linear span of a collection of vectors $v_1, \ldots, v_k$. Then according to~\cite{ToricVarieties} there is an equivalence of categories between the category of toric varieties and toric morphisms and the category of fans and their morphisms. The base of the SYZ fibration for a toric del Pezzo surface is just the fan. This can be recovered, up to a choice of embedding into $\RR^2$, as the dual intersection complex of $D$.

\begin{construction}

Let $D = D_1 \cup D_2 \cup \ldots \cup D_n$ be a cycle of rational curves on a smooth surface $S$ such that $D_i \cap D_j$ is a single point just when $i$ and $j$ differ by 1 mod $n$ and otherwise is empty (in the case $n=2$ we relax this to saying that there are two points in the unique intersection $D_1 \cap D_2$). Then $\Delta_{(S,D)}$ contains precisely one zero-dimensional cell 
$\{0\}$ corresponding to the interior $S \setminus D$. For each component $D_i$, $\Delta_{(S,D)}$
contains a one-dimensional cone with $v_i$ its primitive generator.
Attach the zero-dimensional cone as 0 inside each of these rays. Now introduce a two-dimensional cone in $\Delta_{(S,D)}$ for each intersection point of $D_i \cap D_j$, spanned by $v_i$ and $v_j$. This produces the \emph{dual intersection complex} $\Delta_{(S,D)}$ as a cone complex but it carries more structure. 
We write $B$ to be the underlying topological space of the cone complex
$\Delta_{(S,D)}$. It is homeomorphic to $\RR^2$. We give $B\setminus\{0\}$
an affine manifold structure by defining an affine coordinate chart by embedding
the union of the cones $\RR^{\geq 0} v_{i+1} \oplus \RR^{\geq 0} v_i $ and $\RR^{\geq 0} v_{i} \oplus \RR^{\geq 0} v_{i-1}$ into $\RR^2$ via the relations $v_{i-1} \mapsto (1,0)$, $v_i \mapsto (0,1)$ and $v_{i+1} \mapsto (-1,-D_i ^2)$. This expresses $\Delta_{(S,D)}$ not just as a complex of sets but an affine manifold with singularities (indeed a single singularity at the origin).

\end{construction}

This construction makes perfect sense outside of the toric case, and morally should give the correct base for the mirror, except that it may not embed via an affine map into $\RR^2$. This fits into the general picture that the base of the SYZ fibration should be an affine manifold with singularities, a topological manifold with charts away from codimension two such that the transition functions lie in the affine transformation group $\RR^n \ltimes Sl (\ZZ^n)$. That the dual intersection complex does embed into $\RR^2$ in the toric case is a consequence of the following lemma:

\begin{lemma}[Toric reconstruction]

Let $S$ be a \Lawrence{compact} toric surface and $D$ the toric boundary. Then $(S,D)$ is a Looijenga pair, $\Delta_{(S,D)}$ embeds into $\RR^2$ and any choice of embedding gives the fan for a toric variety isomorphic to $S$.

\begin{proof}

  See~\cite{ToricVarieties} chapter 3.
  
\end{proof}

\end{lemma}

An integral affine manifold $M$ carries a sheaf of integral vector
fields, which we write $\Lambda_M$. On a chart this is isomorphic to the 
constant sheaf with coefficients $\ZZ^n$. 
The transition functions naturally give identifications of these constant
sheaves. In addition, we define $M(\ZZ)$ to be the set of points of $M$
with integer coordinates in some, hence all, integral affine coordinate
chart.

Given a choice of ample divisor $H$ on $S$ toric, the fan $\Delta_{(S,D)}$
carries a natural choice of piecewise linear function up to a choice of global linear function and we denote this $\phi$. Such a function is defined 
by how it changes where it is non-linear. If $\phi_{v_i,v_{i+1}}$ denotes
the linear extension of $\phi$ restricted to $\RR^{\ge 0}v_i+\RR^{\ge 0}v_{i+1}$, then we determine $\phi$ by insisting that
$\phi_{v_i,v_{i+1}}-\phi_{v_{i-1},v_i}=(H.D_i)n_i$, where $n_i$ is a primitive element of the dual space vanishing on $v_i$ and positive on $v_i \RR^{\ge 0} + v_{i+1} \RR^{\ge 0}$.

Let us calculate the mirror in one of the toric cases, $\PP^2$. As we said in the introduction that in the toric case there are no corrections from singularities of the fibration. We will see that the conjectured isomorphism exists.

\begin{example}

  Our first goal is to find the fan. Fortunately we learnt a fan for $\PP^2$ back in infancy: it has one cells generated by $(1,0), (0,1)$ and $(-1,-1)$ as shown in Figure~\ref{P2Picture}.

  \begin{figure}[h]
    \centering
    \begin{tikzpicture}
      \draw (0,0) -- (2,0);
      \draw (0,0) -- (0,2);
      \draw (0,0) -- (-1,-1);
      \filldraw (1,1) circle (1pt);
      \draw (1,0.7) node {$(1,1)$};
      \draw [dotted] (2,1) -- (1,1);
      \draw [dotted] (-1,-1) -- (1,1);
      \draw [dotted] (1,2) -- (1,1);
    \end{tikzpicture}
    \caption{}
    \label{P2Picture}
    \end{figure}
  
  Calculating the product defined above, using the ample divisor
$H$ being the class of a line in $\PP^2$, one finds $\vartheta_{(1,0)} \vartheta_{(0,1)} = z \vartheta_{(-1,-1)}$. As in the basis constructions of~\cite{TheGreedyBasisEqualsTheThetaBasis} these functions morally ought to produce an embedding of the mirror into affine space and their sum should be the super-potential. Doing this we obtain the mirror as being $\mathbb{G}_m^2$ with superpotential $x + y + z / xy$. As expected this is the mirror predicted by other constructions. Calculating the Jacobian ring of the singularity we find that it is $\ZZ [X, z] / \langle X^3 - z \rangle$, isomorphic to the quantum cohomology of $\PP^2$.

\end{example}

Gross and Siebert were influenced initially by the Mumford degeneration and we turn to this for guidance on how to construct a choice of $\phi$ in the non-toric case.This construction will also show how different choices of ample $H$ can give rise to different mirrors. 

We fix the data of all the Looijenga pairs $(S,D)$ we will study, so let us record here a choice of boundary divisor and affine manifold for all the non-toric del Pezzo surfaces. Here we record the boundary divisor in $\PP^2$ before blowing up.
The circles represent points to be blown up, with the exceptional divisor included in the boundary if the circle is red. Since we can't embed
the dual complex as an affine manifold in $\RR^2$, we instead provide a list
of cones in $\RR^2$ and an integral affine isomorphism between two of these
cones. The dual intersection complex is then obtained by identifying these
two cones using this isomorphism.
In the pictures below the cone shaded grey is identified integral linearly with the standard first quadrant, whilst the dotted region is removed entirely.

\begin{center}
\begin{tabular}{>{\centering\arraybackslash}m{1.5cm}|>{\centering\arraybackslash}m{2.8cm}|>{\centering\arraybackslash}m{6cm}|>{\centering\arraybackslash}m{2.8cm}}
  Surface & Boundary & Class of boundary & Dual complex \\
  $dP_5$ \arraybackslash \arraybackslash & \begin{tikzpicture} \draw [red] (0,0) circle (1pt); \draw [red] (1,0) circle (1pt); \draw (0,1) circle (1pt); \draw (0.2,0.2) circle (1pt); \draw [red] (-0.5,0) -- (1.5,0); \draw [red] (0,-0.5)--(0,1.5); \draw [red] (1.5, -0.125) -- (-0.5, 0.375); \end{tikzpicture} & \makecell{$\cup D_i = E_1 \cup (H - E_1 - E_2)$ \\ $ \cup E_2 \cup (H-E_2 - E_3) \cup (H-E_1 - E_4)$}  & \begin{tikzpicture} \fill[lightgray] (0,0) -- (1,0) -- (1,-1); \draw (0,0) -- (1,0); \draw (0,0)--(0,1); \draw (0,0) -- (-1,1); \draw (0,0) -- (-1,0); \draw (0,0) -- (0,-1); \draw[dashed] (0,0) -- (1,-1); \end{tikzpicture} \\
  $dP_4$ & \begin{tikzpicture} \draw [red] (0,0) circle (1pt); \draw (1,0) circle (1pt); \draw (0,1) circle (1pt); \draw (0.5,0.2) circle (1pt); \draw (0.2,0.5) circle (1pt); \draw [red] (-0.5,0) -- (1.5,0); \draw [red] (0,-0.5)--(0,1.5); \draw [red] (-0.5,1.2) -- (1.2,-0.5); \end{tikzpicture} & \makecell{$\cup D_i = E_1 \cup (H - E_1 - E_2) \cup$ \\ $(H- E_3 - E_4) \cup (H - E_1 - E_5)$} & \begin{tikzpicture} \fill[lightgray] (0,0) -- (0,-1) -- (1,-1); \draw (0,0) -- (1,0); \draw (0,0)--(0,1); \draw (0,0) -- (-1,1); \draw (0,0) -- (-1,0); \draw[dashed] (0,0) -- (0,-1); \fill[pattern = dots] (0,0) -- (1,0) -- (1,-1); \end{tikzpicture} \\
  $dP_3$ & \begin{tikzpicture} \draw (0.8,0) circle (1pt); \draw (-0.8,0) circle (1pt); \draw (0.3,0.866) circle (1pt); \draw (0.5,-0.416) circle (1pt); \draw (-0.5,-0.416) circle (1pt); \draw (-0.3,0.866) circle (1pt); \draw [red] (-1.2, -0.416) -- (1.2,-0.416); \draw [red] (-1.1,-0.520) -- (0.066, 1.5); \draw [red] (1.1,-0.52) -- (-0.066, 1.5); \end{tikzpicture} & \makecell{$\cup D_i = (H - E_1 - E_2) \cup$ \\ $(H- E_3 - E_4) \cup (H - E_5 - E_6)$} & \begin{tikzpicture} \fill[lightgray] (0,0) -- (0,-1) -- (-1,-1) -- (-1,0); \draw (0,0) -- (1,0); \draw (0,0)--(0,1); \draw (0,0) -- (-1,1); \draw[dashed] (0,0) -- (-1,0); \fill[pattern = dots] (0,0) -- (1,0) -- (1,-1) -- (0,-1); \end{tikzpicture} \\
  $dP_2$ & \begin{tikzpicture} \draw [red] (0.6,1) .. controls (-1.5, 0) .. (0.6, -1); \draw [red] (0,1)--(0,-1); \draw (-0.975, 0) circle (1pt); \draw (-0.6, 0.4) circle (1pt); \draw (-0.6, -0.4) circle (1pt); \draw (0, 0.3) circle (1pt); \draw (0, -0.3) circle (1pt); \draw (0.38, 0.9) circle (1pt); \draw (0.38, -0.9) circle (1pt); \end{tikzpicture} & \makecell {$\cup D_i = (H - E _1 - E_2) \cup$ \\ $(2H - E_3 - E_4 - E_5 - E_6 - E_7)$} & \begin{tikzpicture} \fill[lightgray] (0,0) -- (-1,0) -- (-1,1); \draw (0,0) -- (1,0); \draw (0,0)--(0,1); \draw [dashed] (0,0) -- (-1,1); \fill[pattern = dots] (-1,0) -- (1,0) -- (1,-1) -- (-1,-1); \end{tikzpicture} 
\end{tabular}
\end{center}

\subsection{The Mumford degeneration}

The construction of the mirror family was inspired by the Mumford degeneration of an Abelian variety to the union of toric varieties. Let us recall the construction of~\cite{AnAnalyticConstructionOfDegeneratingAbelianVarietiesOverCompleteRings}.

\begin{example}

  Let $N\cong \NN^k$ be a lattice, $B \subset N _{\otimes \NN} \RR$ be a lattice polyhedron, $\mathcal{P}$ a lattice polyhedral decomposition of $B$ and $\phi: B \rightarrow \RR$ a strictly convex piecewise linear integral function. One takes the graph over $\phi$ to produce a new polyhedron

\[\Gamma (B, \phi) := \{ (n, r) \in N_\RR \oplus \RR \mid n \in B, r \geq \phi (n)\}\]

This produces a lattice polyhedron unbounded in the positive direction on $\RR$. To construct a family from this we perform a cone construction, let $C (\Gamma)$ be the closure of the cone over $\Gamma (B, \phi)$:
\[
C(\gamma) = \overline{\{(n,r_1,r_2)\in N_{\RR}\oplus\RR\oplus\RR\,|\,
(n,r_1)\in r_2 \Gamma(B,\phi)\}}.
\]
 This carries an action of $\RR^+$ given by translating the second component of $\Gamma(B, \phi)$. The integral points of $C (\Gamma)$ form a graded monoid and we can take $Proj$ of this to produce a variety projective over $\AA^1$. We write this $\PP_{\Gamma (B,\phi)}$. The general fibre is a toric variety, 
whilst the fibre over the origin is a union of toric varieties.

\end{example}

The construction of~\cite{MirrorSymmetryforLogCY1} mimics this in reverse: it claims that the central fibre should be the $n$-vertex, which for $\Lawrence{n>2}$ is a cycle of $n$ copies of $\AA^2$ glued adjacently along their axes, and then attempts to smooth. We cannot hope to have global coordinates as in the toric Mumford degeneration, only local coordinates. To incorporate this data the authors of~\cite{MirrorSymmetryforLogCY1} define a twist of the tangent sheaf which has enough global sections, in particular the function $\phi$ lifts to a section of this bundle.

\begin{definition}
\label{def:sheaf-of-monoids}
  Let $(S,D)$ be a Looijenga pair with associated dual intersection complex $\Delta_{(S,D)}$ and let $\eta: NE (S) \rightarrow M$ be a homomorphism of monoids. We want to construct a multi-valued piecewise linear function on $\Delta_{(S,D)}$ which will bend only along the one-cells. This will be a collection of
piecewise linear functions defined on open subsets of $B$ which differ by linear
functions on overlaps. Such functions are determined by their bends at
one-cells, which are encoded as follows. For a
one-cell $\tau=\RR^{\geq 0} v_i$, choose an orientation $\sigma_+$ and 
$\sigma_-$ of the two two-cells separated by $\tau$ and let $n_\tau$ be the unique primitive linear function positive on $\sigma_+$ and annihilating $\tau$. We want
to construct a representative $\phi_i$ for $\phi$ on $\sigma_+\cup \sigma_-$.  Writing $\phi_+$ and $\phi_-$ for the linear function defined by $\phi_i$ on $\sigma_+$ and $\sigma_-$, the function $\phi_i$ is then defined up to a linear
function by the requirement that
\[ \phi_+ - \phi _- = n_\tau \otimes \eta ([D_i])\]
Such a function is convex in the sense of~\cite{MirrorSymmetryforLogCY1} Definition 1.11, and one says that it is \emph{strictly convex} if $\eta([D_i])$ is not invertible for any $i$. 

This function determines a $M^{gp}\otimes\RR$-torsor $\PP$ as defined 
in~\cite{GrossSiebert}, Construction 1.14, on $B\setminus \{0\}$
which is trivial on each $(\sigma_+\cup\sigma_-)\setminus\{0\}$, i.e.,
is given by $((\sigma_+\cup\sigma_-)\setminus \{0\})\times (M^{gp}\otimes\RR)$.
These trivial torsors are glued on the overlap of two adjacent such sets,
namely on $\RR^{\ge 0} v_{i} \oplus\RR^{\ge 0}v_{i+1}$, via the map
\[ (x,p) \mapsto (x, p + \phi_{i+1} (x) - \phi_i (x))\]
which induce isomorphisms on the monoid of points lying above $\phi$.
This construction is designed to allow us to run the Mumford construction locally, even if we cannot run it globally. 
Write
$\pi: \PP \rightarrow B\setminus\{0\}$, and we write $\mathcal{M}$ for the 
sheaf $\pi_* (\Lambda _{\PP})$, bearing in mind that $\PP$ also has the
structure of an affine manifold. Then there is a canonical exact sequence
  \[ 0 \rightarrow \underline{M^{gp}} \rightarrow \mathcal{M} \xrightarrow{r} \Lambda_{B} \rightarrow 0 \]
  We write $r$ for the second map in this sequence. This will not be mentioned again until we define the canonical scattering diagram so keep this in mind until then.
Furthermore, \cite{GrossHackingSiebert}, Definition 1.16 gives a subsheaf
of monoids $\mathcal{M}^+\subset \mathcal{M}$.
 
\end{definition}

If one performs this construction in the case of a toric variety with the function $\phi$ and pairs it with an ample class one obtains the height function we used in the Mumford degeneration. From this Gross, Hacking and Keel use such a $\phi$ to produce a canonical deformation of the $n$-vertex corresponding to the case where there are no corrections. We introduce the corrected version directly.

\section{Scattering diagrams}

In Mumford's degeneration of an Abelian variety above all the fibres of the SYZ torus fibration on a generic member of the family are smooth. This cannot be the case for a special Lagrangian fibration of a general variety, there must be some singular fibres in the interior and around these singular fibres there will be some monodromy action. This monodromy is an obstruction to a naive product rule $\vartheta_P \vartheta_Q = z^{\phi (P+Q) - \phi(P) - \phi(Q)} \vartheta_{+P+Q}$ being well-defined on the mirror. Thus one needs a
corrected notion of this product, which from symplectic geometry can be done by counting so-called Maslov index two disks as described in~\cite{MirrorSymmetryAndTDualityInTheComplementOfAnAnticanonicalDivisor}. These can
be generated by gluing on Maslov index zero disks onto standard 
Maslov index two disks. Rather than trying to make this symplectic
heuristic precise, we instead are motivated by this picture as follows.
If we stand well away from the singularities and tropicalise these Maslov
index zero disks they appear to be a collection of lines passing out from the origin together with the information of their class. This data is the inspiration for the definition below of a scattering diagram on the base $B$.

\begin{definition}

  Let $B$ be an affine manifold with a single singularity, with 
$B$ homeomorphic to $\RR^2$ with singularity at the origin, so that $B^*:=
B\setminus \{0\}$ is an affine manifold.
Let $\mathcal{M}$ be a locally constant sheaf of Abelian groups
on $B^*$ with a subsheaf of monoids $\mathcal{M}^+\subset\mathcal{M}$
and equipped with a map $r:\mathcal{M} \rightarrow \Lambda_B$.
Let $\mathcal{J}$ be a sheaf of ideals in $\mathcal{M}^+$ with stalk $\mathcal{J}_x$ maximal in $\mathcal{M}^+_x$ for all $x \in B^*$.
Let
$\mathcal{R}$ denote the sheaf of rings locally given by the completion of 
$k[\mathcal{M}^+]$ at $\mathcal{J}$. 
 A \emph{scattering diagram with values in the pair $(\mathcal{M}, \mathcal{J})$ on $B$} is a function $f$ which assigns to each rational ray from the origin an section of the restriction of $\mathcal{R}$ to the ray.
We require the following properties of this function:

  \begin{itemize}
    \item For each $\mathfrak{d}$ one has $f(\mathfrak{d}) = 1$ mod 
$\mathcal{J}|_{\mathfrak{d}}$.
    \item For each $n$ there are only finitely many $\mathfrak{d}$ for which $f(\mathfrak{d})$ is not congruent to $1$ mod $\mathcal{J}|_{\mathfrak{d}}^n$. These $\mathfrak{d}$ are called \emph{walls}.
    \item For each ray $\mathfrak{d}$ and for each monomial $z^{p}$ appearing in $f(\mathfrak{d})$ one has $r(p)$ tangent to $\mathfrak{d}$. A line for which $r(p)$ is a positive generator of $\mathfrak{d}$ for all $p$ with $c_p\not=0$
 is called an \emph{incoming ray}. If instead
$r(p)$ is a negative generator of $\mathfrak{d}$ for all $p$ with $c_p\not=0$,
it is called an \emph{outgoing ray}.
  \end{itemize}
  
  We denote such an object by the tuple $(B, f, \mathcal{M}, \mathcal{J})$. We say that $(B,f,\mathcal{M},\mathcal{J})$ is obtained from $(B,f',\mathcal{M},\mathcal{J})$ by \emph{adding outgoing rays} if for each ray $\mathfrak{d}$ one can write $f(\mathfrak{d}) = f' (\mathfrak{d}) (1 + \sum (c_p z^p))$ where for each monoid element $p$ with $c_p\not=0$ the vector $-r(p)$ is a generator of $\mathfrak{d}$.

\end{definition}

In our case the sheaf of monoids $\mathcal{M}^+$ will be as given in
Definition \ref{def:sheaf-of-monoids}, with the monoid $M$ being
 a finitely generated sharp submonoid of $H_2(S,\ZZ)$
containing $NE(S)$, the monoid generated by effective curves on $S$. Being sharp means the only invertible element of $M$ is the identity element, and so the maximal ideal is just the complement of the identity. We introduce the choice of scattering diagram, the \Lawrence{\emph{canonical scattering diagram},} and then motivate and define each term appearing. A ray $\mathfrak{d} = (a v_i + b v_{i+1}) \RR^{\geq 0}$ specifies a blow up of $S$ given by refining the fan until $\mathfrak{d}$ is a one-cell and all the two-cells are integrally isomorphic to $(\RR^{\geq 0})^2$. The ray $\mathfrak{d}$ then corresponds to a component $C$ in the inverse image of $D$. To this ray then we assign the power series
\[f(\mathfrak{d}) := \exp \left[ \sum_\beta k_{\beta} N_\beta z^{\eta (\pi_* (\beta)) - k_\beta m_{\mathfrak{d}}'}\right],\]
where $m_{\mathfrak{d}}'$ is the unique lift of a primitive outward
pointing tangent vector $m_{\mathfrak{d}}$ to $\mathfrak{d}$ not contained
in the relative interior of $\mathcal{M}|_{\mathfrak{d}}$.
The number $N_{\beta}$ counts the number of relative curves mapping to $(S,D)$ tangent to $C$ to maximal order $k_\beta$ at a single point as outlined in~\cite{MirrorSymmetryforLogCY1} at the beginning of section 3.
 By the Riemann-Roch formula this is of dimension
\[ dim \: S - 3 - K_S. \beta - k_\beta + 1 = 0.\]
By~\cite{GromovWittenInvariantsInAlgebraicGeometry} this produces a Gromov-Witten invariant. This construction may also be run using logarithmic Gromov-Witten invariants rather than these blow ups, see~\cite{LogarithmicGromovWittenInvariants} for the definition, and this will be explored in future work of Gross and Siebert.

Conjecturally this encodes the glueing data for a generating set of the open Gromov-Witten invariants. It should be possible to recreate the entire Gromov-Witten theory of $S$, both open and closed from this data but this problem is very difficult.

\subsection{Broken lines}

To encode the Maslov index two disks themselves we collapse them onto a skeleton in the base of the SYZ fibration, having pushed all the singularities to the origin. Then following~\cite{AmoebasOfAlgebraicVarietiesAndTropicalGeometry} the disks appear as balanced tropical curves together with a mark of the class that they lie in. The skeleta of these disks are trivalent graphs with one leg ending at the origin and the other two legs passing to infinity. The definition below defines how one of these legs behave, the full picture only arising once we have discussed pairs of pants and even then we suppress the third leg which by necessity ends at origin.

\begin{definition}

  A \emph{broken line from $v\in B(\ZZ)$ to $P \in B$} on a scattering diagram $(B, f, R, J)$ is a choice of piecewise linear function $l: \RR^{\leq 0} \rightarrow B$ and a map $m: \RR^{\leq 0} \rightarrow \prod_{t\in\RR^{\leq 0}}
\mathcal{R}_{l(t)}$ such that the following hold:

  \begin{itemize}
  \item $l$ has only finitely many points where it is non-linear and these only occur where $l$ maps into a rational ray of $B$.
  \item $l (0) = P$.
  \item $l(t)$ lies in the same cone as $v$ and is parallel to $v$ for all $t$ sufficiently negative.
  \item $l$ does not map an interval to a ray through the origin.
    \item $m(t)\in \mathcal{R}_{l(t)}$ for all $t$ and is a monomial in
this ring, written as $c_t z^{m_t}$. Further, on each domain of linearity
of $l$, $m(t)$ is given by a section of $\mathcal{R}$ pulled back to 
this domain of linearity.
  \item For $t$ very negative, $m_t\in \mathcal{M}^+_{l(t)}$ is the
unique element not lying in the interior of this monoid satisfying
$r(m_t) = v$.
  \item $r(m(t_0)) = -{ \partial l \over \partial t}\big|_{t=t_0}$ wherever $l$ is linear.
  \item $m$ only changes at those points where $l$ is non-linear.
  \item Let $t \in \RR_{\leq 0}$ be a point where $l$ is non-linear. We write $\partial(l_+)$, $\partial(l_-)$, $m_+$ and $m_-$ for the values of $\partial l \Lawrence{/} \partial t$ and $m$ on either side of $t$. Suppose that $l(t)$ lies on a ray $\mathfrak{d}$ with primitive normal vector $n_{\mathfrak{d}}$, negative on $\partial (l_-)$. Then $m_+$ is a monomial term of $m_- f_{\mathfrak{d}}^{\langle n_{\mathfrak{d}}, \partial (l_-)\rangle}$.
\end{itemize}

\end{definition}

There is more technical content in~\cite{MirrorSymmetryforLogCY1} exploring how to deform the Mumford construction to produce a formal smoothing of the $n$-vertex. We do not concern ourselves with those results at the moment since we now have enough to define the multiplication rule on the $\vartheta$-functions. In ~\cite{GrossHackingSiebert} Gross, Hacking and Siebert introduce an entirely abstract construction of the $\vartheta$-functions. 

The key data we will need is the count of pairs of pants, which are expressed in terms of broken lines as pairs of broken lines $(l_P, m_P)$  and $(l_Q, m_Q)$ from $P$ and $Q$ respectively to an irrational point near to $R$ such that 
$(\partial l_P/\partial t)|_{t=0} + (\partial l_Q/\partial t)_{t=0} = R$. Let $T_{P,Q \rightarrow R}$ denote the set of such pairs.

We have seen on page~\pageref{ToricThetaProduct} how to define a product on the $\vartheta$-functions in the toric case, so let us now study the general case. The outcome of our conversation about Hochschild cohomology was that the ring of regular functions on the mirror is the degree zero part of symplectic cohomology of the original variety. Following~\cite{CanonicalBasesForClusterAlgebras} for each integral point $P$ in $B$ introduce a symbol $\vartheta_P$. As a $k$-vector space the ring $QH (\check{W})$ is freely generated by the $\vartheta_P$. We take the content of~\cite{MirrorSymmetryforLogCY1} Theorem 2.34 as the definition of the product of $\vartheta_P$ and $\vartheta_Q$, so the product $\vartheta_P \vartheta_Q$ is equal to

\[ \sum_{R} \sum_{((l_P, m_P), (l_Q, m_Q)) \in T_{P,Q \rightarrow R}} m_P (0) m_Q (0) \vartheta_R\]

This is analogous to the approach of~\cite{GrossHackingSiebert}. A key result of~\cite{MirrorSymmetryforLogCY1} is that for fixed order $J^n$ and for a consistent scattering diagram in the sense of~\cite{MirrorSymmetryforLogCY1}, Definition 2.26, this does not depend on the choice of irrational point near $R$. This consistency property will hold in particular for the canonical scattering diagrams considered below. This sum need not terminate and indeed will in general only produce a power series. However, in the case that $D$ supports an ample divisor, this power series will in fact be a polynomial. Our goal in what follows is to limit the terms which can occur and describe a generating basis of the $\vartheta$-functions.

\subsection{The combinatorics of scattering diagrams}

We now explain how to use these functions to embed the mirror into affine space. At the moment it is not clear that the relations between the $\vartheta$ functions should be algebraic. Let us explain how to exploit the structure of a scattering diagram to prove algebraic convergence. To do this we require an example to work with, for simplicity the del Pezzo surface $dP_5$.

\begin{example}Figure~\ref{fig:dP5Structure} below gives a choice of boundary divisor,
where circles denote blown-up points and red lines denote components of $D$. 
\begin{figure}[h]
  \centering
  \begin{tikzpicture}
    \draw [red] (0,0) circle [radius=0.1];
    \draw [red] (4,0) circle [radius=0.1];
    \draw [dotted] (0,4) circle [radius=0.1];
    \draw [dotted] (1,1.8) circle [radius=0.1];
    \draw [red] (-1, 0) -- (5, 0);
    \draw [red] (0, -1) -- (0, 5);
    \draw [red] (-1, 3) -- (5, -0.6);
  \end{tikzpicture}
  \caption{}
  \label{fig:dP5Structure}
\end{figure}
From this we construct the base of the scattering diagram as in Figure~\ref{fig:dP5DualComplex}, where the grey section is glued isomorphically to the standard upper right quadrant

\begin{figure}[h]
\centering
\begin{tikzpicture}
\path[fill=lightgray] (0,0) -- (2,0) -- (2,-2);
\draw [dashed] (0,0) -- (2,0);
\draw (0,0) -- (0,2);
\draw (0,0) -- (-2,2);
\draw (0,0) -- (-2,0);
\draw (0,0) -- (0,-2);
\draw [dashed] (0,0) -- (2,-2);
\draw (1,1) node {$1$};
\draw (-0.5,1) node {$2$};
\draw (-1,0.5) node {$3$};
\draw (-1,-1) node {$4$};
\draw (0.5,-1) node {$5$};
\end{tikzpicture}
\caption{}
\label{fig:dP5DualComplex}
\end{figure}

In this case one can construct the entire scattering diagram: it is finite and only non-trivial on the one-cells. We will not do this since in future examples the scattering diagram will be non-trivial on a dense set of rays.

Each point of $B$ defines a curve class as via $a \; v_i + b \; v_{i+1} \mapsto a D_i + b D_{i+1}$. Define the piecewise linear function $E$ on $B$ given by $-K_S \cdot (a D_i + b D_{i+1})$. From the Fano condition this defines a strictly convex piecewise linear function on the space $B$ which is positive everywhere
but at the origin. The function $E$ extends to a linear function on the tangent space at each point in the interior of a maximal cell. This allows us to evaluate
$E$ on tangent vectors to broken lines away from one-cells. 

In fact, given a broken line, $E(\partial l/\partial t)$ is increasing
as a function of $t$. Indeed,
at a point $t$ where $l$ is non-linear by definition the tangent vector
changes by a positive generator of the ray $\mathfrak{d}$, and $E$ is 
positive on this generator.
Similarly if $l$ crosses from one maximal cell of $B$ to another then by
convexity of $E$ the value of $E (\partial l / \partial t)$ increases. 

We use this as follows.
Suppose that two broken lines $(l_1, m_1)$ and $(l_2. m_2)$ from $v_1$ and $v_2$ combine to form a pair of pants at $P$. Then by definition we must have 
\[
(\partial l_1/\partial t)|_{t=0} + (\partial l_2/\partial t)|_{t=0} + P = 0.
\]
Since $E(P)\ge 0$, and $-E(v_i)\le E((\partial l_i/\partial t)|_{t=0})$, we
see that
$E(v_1)$ and $E(v_2)$ control how many new maximal cells $l_1$ and $l_2$ can enter, and how many times they can bend.

 Let us begin by understanding what could possibly contribute to the product $\vartheta_{(-1,0)} \vartheta_{(0,1)}$. Let us number the two-cells anti-clockwise starting at the upper right cell as in Figure~\ref{fig:dP5DualComplex}.
  
  Let $l_1$ be a broken line coming from $(-1,0)$ and $l_2$ a broken line from $(0,1)$. We have that $E((-1,0)) = E ((0,1)) = 1$. This immediately implies
that $0 \le E(P)\le 2$. Further, if $E(P)=2-n$, then the total number
of bends and crossings of one-dimensional cells of $l_1$ and $l_2$ is $n$.
First note that $n=0$ does not occur: since $v_1$ and $v_2$ are in different
cones, at least one of the broken lines must cross a one-cell.
If $n=1$, then we cross one one-cell and further bends or crossings can
occur. There is then only one possibility, 
the obvious pair of pants as described in Figure~\ref{Unbent} contributing to $\vartheta_{(-1,1)}$. 
Finally let us study the case $n=2$. In this case, $E(P)=0$, hence $P=0$,
and we can choose a point near the origin, say in the second two-cell.
If $l_2$ enters the grey region as detailed in Figure~\ref{Transition} it emerges moving in the direction $(-1,0)$ and so has $E(\partial l / \partial t) = 1$. 
Therefore in order to reach the second cell, it must cross two more one-cells,
and hence it can't contribute to a pair of pants. Thus $l_2$ must start
in cell 2 and never leaves it.

A similar argument shows that $l_1$ can only pass from cell 3 to cell 2,
and either $l_1$ or $l_2$ bends at a ray $\RR^{\ge 0}v$ with $E(v)=1$
and $v\in B(\ZZ)$. There is only one choice for such a ray, namely
the one-cell separating cells 2 and 3. See Figure~\ref{Bend}.
\label{combinatorialfiniteness}
\end{example}

In general, there will only be a finite number of rays with bounded $E$
and only a finite combination of bends. Thus in more complicated cases
this analysis can be implemented in a computer algebra package via the algorithm outlined in Algorithm~\ref{ThetaAlg} and a non-optimised version written in Python is available from me on request.

\begin{algorithm}
\caption{Theta relations algorithm}\label{ThetaAlg}
\begin{algorithmic}[1]
\State {\textbf{object} \textsc{AffineManifold}}
\LeftComment{Define a data structure consisting of a fan $\Sigma$, a piecewise linear function $\phi$ and a collection of integer points of $\Sigma$, $\Sigma_{\ZZ}$.}
\State {\textbf{object} \textsc{BrokenLine}}
\LeftComment {Define a data structure consisting of an initial direction $v$, a monomial, a tangent direction thought of as $\partial l/ \partial t \mid _{t=0}$ and a list of points $P_i$ generating rays along which the broken line bent}
\Function{ProductFormula} {$l_1$, $l_2$, affinemanifold}
\LeftComment {Define a function of two broken lines with tangent directions $v_1$ and $v_2$ in an affine manifold, affinemanifold which finds all possible terms appearing in the product $\vartheta_{v_1} \vartheta_{v_2}$}
\State{\textbf{int} maxE := $\phi (v_1) + \phi (v_2)$}
\State{\textbf{def} endpoints := The set of points in $\Sigma_{\ZZ}$ such that $\phi(P) \leq maxE$.}
\For {P in endpoints}
\If {$l_1$ and $l_2$ could pass through P and $v_1 + v_2 + P =0$}
\State {\textbf{print} the data of how $l_1$ and $l_2$ have bent}
\EndIf
\If {$l_1$ could pass through $P$}
\State {\textbf{def} newline := $l_1$ but now also scattering off of $\RR^{\geq 0} P$ as well}
\State {ProductFormula (newline, $l_2$)}
\EndIf
\If {$l_2$ could pass through $P$}
\State {\textbf{def} newline := $l_2$ but now also scattering off of $\RR^{\geq 0} P$ as well}
\State {ProductFormula ($l_1$, newline)}
\EndIf
\EndFor
\EndFunction
\State {\textbf{def} affinemanifold := AffineManifold( \ldots )}
\LeftComment {Define the base of the appropriate scattering diagram}
\State {\textbf{def} $L_1$ := BrokenLine ( \ldots )}
\LeftComment {Define $L_1$ to be a broken line, not scattering off of any rays with initial direction $v_1$}
\State {\textbf{def} $L_2$ := BrokenLine ( \ldots )}
\LeftComment {Define $L_2$ to be a broken line, not scattering off of any rays with initial direction $v_2$}
\State{ProductFormula ($L_1$, $L_2$, affinemanifold)}
\end{algorithmic}
\end{algorithm}

\begin{figure}[h]
    \centering
    \begin{subfigure}[b]{0.3\textwidth}
        \begin{tikzpicture}
          \path[fill=lightgray] (0,0) -- (2,0) -- (2,-2);
          \draw [dashed] (0,0) -- (2,0);
          \draw (0,0) -- (0,2);
          \draw (0,0) -- (-2,2);
          \draw (0,0) -- (-2,0);
          \draw (0,0) -- (0,-2);
          \draw [dashed] (0,0) -- (2,-2);
          \draw [red, -latex] (-2,1) -- (-1,1) ;
          \draw [red, -latex] (-1,2) -- (-1,1) ;
          \draw [red, -latex] (0.01,0.01) -- (-0.99,1.01) ;
      \end{tikzpicture}
      \caption{}
      \label{Unbent}
    \end{subfigure}
    \begin{subfigure}[b]{0.3\textwidth}
      \begin{tikzpicture}
          \path[fill=lightgray] (0,0) -- (2,0) -- (2,-2);
          \draw [dashed] (0,0) -- (2,0);
          \draw (0,0) -- (0,2);
          \draw (0,0) -- (-2,2);
          \draw (0,0) -- (-2,0);
          \draw (0,0) -- (0,-2);
          \draw [dashed] (0,0) -- (2,-2);
          \draw [red] (-2,0.5) -- (-0.5,0.5) ;
          \draw [red, -latex] (-0.5,0.5) -- (-0.5,1) ;
          \draw [red, -latex] (-0.5,2) -- (-0.5,1) ;
      \end{tikzpicture}
        \caption{}
        \label{Bend}
    \end{subfigure}
    \begin{subfigure}[b]{0.3\textwidth}
            \begin{tikzpicture}
          \path[fill=lightgray] (0,0) -- (2,0) -- (2,-2);
          \draw [dashed] (0,0) -- (2,0);
          \draw (0,0) -- (0,2);
          \draw (0,0) -- (-2,2);
          \draw (0,0) -- (-2,0);
          \draw (0,0) -- (0,-2);
          \draw [dashed] (0,0) -- (2,-2);
          \draw [dotted] (0,0) -- (-2,1);
          \draw [dotted] (0,0) -- (2,2);
          \draw [red] (0.5,2) -- (0.5,0);
          \draw [red, -latex] (0.5,-0.5) -- (-2,-0.5);
      \end{tikzpicture}
        \caption{}
        \label{Transition}
    \end{subfigure}
    \caption{}
    \label{BrokenLines}
\end{figure}

Doing this systematically for each pair $\vartheta_{v_{i-1}} \vartheta_{v_{i+1}}$ we can find schematic relations for the products which highlight which relative invariants we need to find. Filling out the relations for $dP_5$ we find the following relations.

\begin{align}
\vartheta_{(1,0)} \vartheta_{(-1,1)} &= g_1\vartheta_{(0,1)} + f_1 \\
\vartheta_{(0,1)} \vartheta_{(-1,0)} &= g_2\vartheta_{(-1,1)} + f_2 \\
\vartheta_{(-1,1)} \vartheta_{(0,-1)} &= g_3\vartheta_{(-1,0)} + f_3 \\
\vartheta_{(-1,0)} \vartheta_{(1,0)} &= g_4\vartheta_{(0,-1)} + f_4\\
\vartheta_{(0,-1)} \vartheta_{(0,1)} &= g_5\vartheta_{(1,0)} + f_5
\end{align}

The terms $f_i$ count curves meeting $D_i$ at a single point with tangency order one together with some monomials coming from broken lines looping around the origin.

As in the work of Hulya Arguz~\cite{HulyaThesis} the easy case is when the boundary consists of at least four components and the corresponding algebra is a quadratic algebra. The three component case is tractable but to compute the two component case $dP_2$ it is best to blow up the two intersection points of the two divisors. This refines the dual complex but keeps the interior $S \setminus D$ invariant. As a consequence, this changes the mirror family
because the base $\Spec k[P]$ increases in dimension. However, the mirror
family to the original surface can be found by restricting to a subfamily,
see \cite{MirrorSymmetryforLogCY1}, Section 6.2 for details.
In keeping with our previous notation we denote the exceptional curves over these by $E_8$ and $E_9$. The choice of dual complex now becomes:
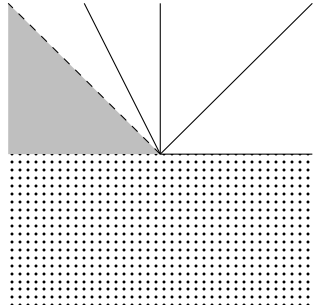
\[\begin{tikzpicture} \fill[lightgray] (0,0) -- (-2,0) -- (-2,2); \draw (0,0) -- (2,0); \draw (0,0) -- (-1,2); \draw (0,0) -- (2,2); \draw (0,0)--(0,2); \draw [dashed] (0,0) -- (-2,2); \fill[pattern = dots] (-2,0) -- (2,0) -- (2,-2) -- (-2,-2); \end{tikzpicture}\]
\Lawrence{we re-label the theta functions by the corresponding cohomology class, so $\vartheta_{D_i}$ for $\vartheta_i$, further we write $\vartheta_{nD_i}$ for the function $\vartheta_{nv_i}$. With these labels we can apply the techniques of Example~\ref{combinatorialfiniteness} to show that the possible terms appearing in the relations between the theta functions become:
\begin{align}
\vartheta_C \vartheta_L &= r_1 \vartheta_{E_8} + r_2 \vartheta_{E_9} + r_3 + c_1(E_8) \label{formulaestart}\\
\vartheta_{E_8} \vartheta _{E_9} &= r_4 \vartheta_{3L}  + \vartheta_{2L} c_1 (L) + \vartheta_{L} c_2 (L) + c_3 (C) + r_5 \nonumber \\
& \quad + \vartheta_C c_2 (C) + \vartheta_{2C} c_1 (C) + r_6 \vartheta_{3C} \\
\vartheta_{C}^3 &= \vartheta_{3C} + c_1 (L) \vartheta_{C} + r_7 + c_1 (E_9) + r_8 + r_9 \vartheta_{E_8} + r_{10} \vartheta_{E_9}\\
\vartheta_L^3 &= \vartheta_{3L} + c_1 (C) \vartheta_L + r_{11} + c_1 (E_8) +  r_{12} + r_{13} \vartheta_{E_9} + r_{14} \vartheta_{E_8}\\
\vartheta_{C}^2 &= \vartheta_{2C} + r_{15} \vartheta_{C} + c_1 (L) \\
\vartheta_{L}^2 &= \vartheta_{2L} + r_{16} \vartheta_{L} + c_1 (C) \label{formulaeend}
  \end{align}
where the functions $r_i$ come from the combinatorics of broken lines which don't bend. The monomials which they carry come from the fact that the values of $\phi (\partial l/\partial t)$ are different at the start and end of the broken lines. The functions $c_n (D)$ are counts of broken lines bending due to curves tangent to a class $D$ at a single point to order $n$. It is this finite list of open Gromov-Witten invariants which we must calculate.} To calculate these invariants we have two options, either one could directly compute the moduli space and find the degree of the virtual fundamental class or one can use properties of the scattering diagram. Since we have not yet explained the role monodromy plays in the scattering diagram we will do the latter since it forms an important tool in the theory.

\section{Explicit Gromov-Witten counts and Monodromy}
\label{sec:explicit-GW}

We cannot apply the Mumford construction globally in the non-toric case since the function $\phi$ is not global but we can apply it locally to each face. To do so if $\tau$ is a one- or two-dimensional face of $\Delta_{(S,D)}$, we write $\Delta_{(S,D)}/\tau$ for the localised fan at $\tau$, i.e., the fan given by adding $\RR\tau$ to each cone containing $\tau$ and
removing cones of $\Delta_{(S,D)}$ not containing $\tau$. Then after choosing a representative for $\phi$, $\phi_{\tau}$, we have a piecewise linear function single-valued on $\Delta_{(S,D)}/\tau$ and one can consider the graph $\Gamma (\Delta_{(S,D)}/\tau, \phi_{\tau})$. If $\tau$ lies as the one-cell separating two two-cells, $\sigma_+$ and $\sigma_-$, let $\phi_{\sigma_\pm}$ be the linear
extensions of $\phi_{\tau}$ restricted to the maximal cones $\sigma_{\pm}$.
Then we have canonical inclusions of graphs and an equality
\begin{equation}
\label{eq:Gamma-rel}
\Gamma(\Delta_{(S,D)}/\tau, \phi_{\tau}) = \Gamma(\Delta_{(S,D)}/\sigma_+, \phi_{\sigma_+}) \cap \Gamma (\Delta_{(S,D)}/\sigma_-, \phi_{\sigma_-})
\end{equation}
All \Lawrence{the monoids $\Gamma(\Delta_{(S,D)}/\rho, \phi_{\rho})$} have actions by $M$, giving rings 
$k[\Gamma(\Delta_{(S,D)}/\rho),\phi_{\rho}]$ with $k[M]$-algebra structures,
and different choices of representative for $\phi$ lead to canonically 
isomorphic $k[M]$-algebras. Choose a monomial ideal $I\subset M$ such that
$A_I:=k[M]/I$ is Artinian, and write
\[
k[\Gamma(\Delta_{(S,D)}/\rho),\phi_{\rho}]_I:=
k[\Gamma(\Delta_{(S,D)}/\rho),\phi_{\rho}]\otimes_{k[M]} A_I.
\]
Let 
\[
U_{\rho}=\Spec k[\Gamma(\Delta_{(S,D)}/\rho),\phi_{\rho}]_I.
\]
Then the relation \eqref{eq:Gamma-rel} leads to natural open inclusions
$U_{\sigma_{\pm}}\subset U_{\tau}$, which gives us gluing data for a scheme.
The paper~\cite{MirrorSymmetryforLogCY1} gives explicit formulae for the rings involved and proves that this gluing data does indeed satisfy the descent condition. This is the strict analogue of the Mumford degeneration, but does not contain any data about the instanton corrections. In particular, with the above
gluing, it is not obvious that the glued scheme will carry enough
global functions to be embedded into an affine variety. For example,
a naive approach to finding global functions is to attempt to extend
monomial functions from $U_{\rho}$, but this frequently
doesn't work precisely because of monodromy. The authors of~\cite{MirrorSymmetryforLogCY1} correct the gluing using 
automorphisms of the rings $k[\Gamma (\Delta_{(S,D)}/\sigma, \phi_{\sigma})]_I$ 
for the two-cells $\sigma$ and taking a limit over all $I$.

Suppose that we have a scattering diagram $ (f, \Delta_{(S,D)}, R, J)$ and fix a thickening ideal $I$. Let $\sigma$ be a two-cell with boundary one-cells
$\tau_+$, $\tau_-$. Choose points $P_+,P_-\in\sigma$ very close to $\tau_+$,
$\tau_-$ respectively. Choose a path $\gamma$ connecting $P_+$ to $P_-$.
Such a path selects a primitive normal vector $n_{\mathfrak{d}}$ to each ray through the origin in $\sigma$, pairing negatively with the tangent vector of the path. By assumption there are only finitely many rays $\mathfrak{d}$ passing through the origin such that $f(\mathfrak{d})$ is not 1 mod $I$. To each such ray take $\theta_{\mathfrak{d}}$ to be the automorphism defined on a toric monomial $z^p$, $p\in \Gamma(\Delta_{(S,D)}/\sigma,\phi_{\sigma})$ by
\[\theta_{\mathfrak{d}} (z^p) = z^p f(\mathfrak{d}) ^{\langle n_{\mathfrak{d}}, r(p)\rangle}\]
Now we define the automorphism $\theta_{\sigma}$ to be the ordered composition of the $\theta_{\mathfrak{d}}$ as one moves along $\gamma$. This yields
an automorphism of $U_{\sigma}$. In addition, the definition of $U_{\tau}$
is modified slightly using $f(\tau)$, see \cite{MirrorSymmetryforLogCY1},
Equation (2.7). Thus we obtain a modification of the $U_{\tau}$'s and 
a modification of the gluing inclusions $U_{\tau_+}\supset U_{\sigma} \subset
U_{\tau_-}$. This gives a new glued scheme which depends on the
scattering diagram and the ideal $I$. Consistent scattering diagrams then have the property
that the associated rings have a natural basis of regular functions $\{\vartheta_p\,|\,
p\in\Delta_{(S,D)}(\ZZ)\}$ which yield an embedding in an affine scheme
over $\Spec A_I$. One of the main theorems of \cite{MirrorSymmetryforLogCY1},
namely Theorem 3.8, states the canonical scattering diagram is consistent. Importantly there is an algorithmic way of computing consistent
scattering diagrams in certain situations, and a way of reducing the
calculation of the canonical scattering diagram to these situations.

So for the moment, consider a case where $\Delta_{(S,D)}$ is in fact affine isomorphic to $\RR^2$ and $\phi$ is single-valued. Then the sheaf $\mathcal{M}^+$
extends across $0$, with the stalk at $0$ being $\Gamma(\Delta_{(S,D)},\phi)$.
In this construction $\mathcal{R}_0$ is the completion of $k[\mathcal{M}^+_0]$ at its maximal
monomial ideal. All of the automorphisms $\theta_{\sigma}$ act on $\mathcal{R}_0$. Our construction so far has not attempted to incorporate the monoid ring $\mathcal{R}_0$. Take a cyclic ordering of the two-cells $\sigma_1, \ldots \sigma_n$ and let $\theta$ be the composite $\theta_{\sigma_n} \circ \ldots \circ \theta_{\sigma_1}$ considered as a map $\mathcal{R}_0 \rightarrow \mathcal{R}_0$. Then the 
scattering diagram is consistent if the map $\theta$ is the identity, see
\cite{ATropicalViewOfLandauGinzburgModels} and 
\cite{MirrorSymmetryforLogCY1}, Theorem 3.30. Let us give an example calculation.

\begin{example}
  Any toric variety relative to its toric boundary has trivial canonical scattering diagram. Therefore we will find no interesting examples here. Let us give an example motivated instead by $dP_5$: we will see shortly the connection.

  Our initial scattering diagram is $\RR^2$ with two incoming rays $(1 + z^{(1,0)})$ and $(1+z^{(0,1)})$:
  
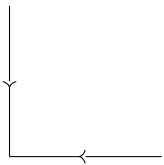
\[\begin{tikzpicture} \draw[-<] (0,0) -- (1,0); \draw[-<] (0,0) -- (0,1); \draw (1,0)--(2,0); \draw (0,1) -- (0,2);\end{tikzpicture}\]
  Let us calculate the composition $\theta$ along a circle anti-clockwise around the origin starting at $(-1,-2)$ acting on $z^{(1,0)}$ and $z^{(0,1)}$. The first wall we meet is $\mathfrak{d}_1=\RR_{\ge 0}(1,0)$. Applying the above definition of the automorphism
$\theta$:
  \[\theta_{\mathfrak{d}_1} (z^{(1,0)}) = z^{(1,0)} (1 + z^{(1,0)}) ^{\langle (1,0),(0,-1) \rangle} = z^{(1,0)}\]
  \[\theta_{\mathfrak{d}_1} (z^{(0,1)}) = z^{(0,1)} (1 + z^{(1,0)}) ^{\langle (0,1),(0,-1) \rangle} = z^{(0,1)}(1 + z^{(1,0)})^{-1}\]
  Continuing along our path we meet the wall $\mathfrak{d}_2=\RR_{\ge 0}(0,1)$. Here $\theta_{\mathfrak{d}_2}$ is given by:
  \[\theta_{\mathfrak{d}_2} (z^{(1,0)}) = z^{(1,0)} (1 + z^{(0,1)}) ^{\langle(1,0),(1,0) \rangle} = z^{(1,0)}(1+z^{(0,1)})\]
  \[\theta_{\mathfrak{d}_2} (z^{(0,1)}) = z^{(0,1)} (1 + z^{(0,1)}) ^{\langle(0,1),(1,0) \rangle} = z^{(0,1)}\]
  Therefore the composition $\theta$ sends $z^{(1,0)}$ to $z^{(1,0)}(1 + z^{(0,1)})$ and $z^{(0,1)}$ to $z^{(0,1)} (1 + z^{(1,0)} + z^{(1,1)})^{-1}$. To first order this looks as if we are missing outgoing rays with monomials $(1+z^{(1,0)})$ and $(1+z^{(0,1)})$. Let us add those in and then repeat this calculation.
  
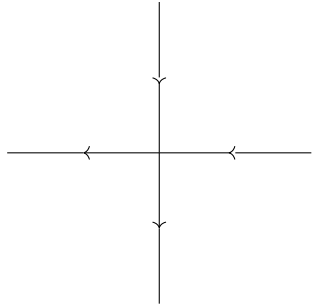
\[\begin{tikzpicture} \draw[-<] (0,0) -- (1,0); \draw[-<] (0,0) -- (0,1); \draw (1,0)--(2,0); \draw (0,1) -- (0,2);\draw[->] (0,0) -- (-1,0); \draw[->] (0,0) -- (0,-1); \draw (-1,0)--(-2,0); \draw (0,-1) -- (0,-2);\end{tikzpicture}\]
  Taking again a path anti-clockwise starting at $(-1,-2)$ we first meet the wall $\mathfrak{d}_3=\RR_{\ge 0}(0,-1)$. The associated $\theta_{\mathfrak{d}_3}$ acts by:
    \[\theta_{\mathfrak{d}_3} (z^{(1,0)}) = z^{(1,0)} (1 + z^{(0,1)}) ^{\langle(1,0),(-1,0) \rangle} = z^{(1,0)}(1+z^{(0,1)})^{-1}\]
  \[\theta_{\mathfrak{d}_3} (z^{(0,1)}) = z^{(0,1)} (1 + z^{(0,1)}) ^{\langle(0,1),(-1,0) \rangle} = z^{(0,1)}\]
  The next wall is $\mathfrak{d}_1$, which again acts as before. The third wall we meet is $\mathfrak{d}_2$, again acting as before. The final wall we encounter is $\mathfrak{d}_4=\RR_{\ge 0}(-1,0)$, where the associated $\theta$ is given by:
  \[\theta_{\mathfrak{d}_4} (z^{(1,0)}) = z^{(1,0)} (1 + z^{(1,0)}) ^{\langle (1,0),(0,1) \rangle} = z^{(1,0)}\]
  \[\theta_{\mathfrak{d}_4} (z^{(0,1)}) = z^{(0,1)} (1 + z^{(1,0)}) ^{\langle (0,1),(0,1) \rangle} = z^{(0,1)}(1 + z^{(1,0)})\]
  One easily calculates the composite of all of these and one finds that to second order the composite acts by $z^{(1,0)} \mapsto z^{(1,0)} (1 + z^{(1,1)} + \ldots)$ so it appears that we are missing an outgoing ray carrying monomial $(1 + z^{(1,1)})$. Introducing such a ray to obtain a new scattering diagram
  
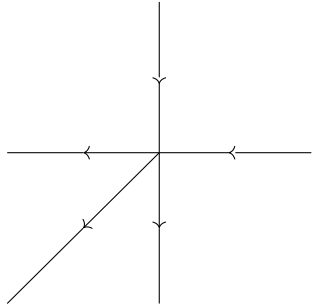
\[\begin{tikzpicture} \draw[-<] (0,0) -- (1,0); \draw[-<] (0,0) -- (0,1); \draw (1,0)--(2,0); \draw (0,1) -- (0,2);\draw[->] (0,0) -- (-1,0); \draw[->] (0,0) -- (0,-1); \draw (-1,0)--(-2,0); \draw (0,-1) -- (0,-2);\draw[->] (0,0) -- (-1,-1); \draw (-1,-1) -- (-2,-2);\end{tikzpicture}\]
 Now one calculates $\theta$ is the identity, so we have finally constructed a consistent scattering diagram.
  
\end{example}

The term-by-term nature of the above construction seems to suggest that in the background lurks a general procedure. The Kontsevich-Soibelman lemma of~\cite{AffineStructuresAndNonArchimedeanAnalyticSpaces} proves that from any scattering diagram embedded in $\RR^2$ (i.e,  one with no singularities at the origin) one can construct a consistent scattering diagram by adding in only outgoing rays. In particular one can perform this correction order-by-order as in our example making it computable. This is explicitly described in the following lemma.

\begin{lemma}[The Kontsevich-Soibelman lemma]

  Let $(B,f,R,J)$ be a scattering diagram with $B \cong \RR^2$ as affine manifolds. Let $P$ be a point of $B$ contained inside a two-cell $\sigma$. The ordered composite of the functions $\theta$ along a path $\gamma$ around the origin with $\gamma(0) = \gamma(1) = P$ gives rise to an automorphism $\theta_{\gamma}$ of the stalk of sheaf $\mathcal{R}$ at zero. There exists another scattering diagram $(B, f', R, J)$ such that the corresponding automorphism $\theta_{\gamma}$ is the identity and such that for each ray $\mathfrak{d}$ one has
\[ f' (\mathfrak{d}) = f(\mathfrak{d}) (1 + \sum a_i z^{m_i})\]
where the $z^{m_i}$ are all outgoing.

\end{lemma}

A key result of Gross, Hacking and Keel, Theorem 3.25 of~\cite{MirrorSymmetryforLogCY1} explores how to flatten the base of the scattering diagram by using \emph{toric models}. A toric model of $(S,D)$ is a choice of birational morphism $\phi:(S, D) \rightarrow (S',D')$ obtained by blowing down curves in the interior of $S$ meeting $D$ in precisely one point such that $S'$ is toric and $D'$ the toric boundary. This restricts to an isomorphism $D \cong D'$ and so induces a map of dual complexes which ``flattens'' the singularities in the sense that the ensuing dual intersection complex is isomorphic to $\RR^2$ as an affine manifold. In the philosophy of Gross Hacking and Keel it pushes the singularities out to infinity. We reincorporate the singularities by instead including them into the scattering diagram as rays carrying expressions $1+z^{[E_i]+\phi(v_i)}$ 
where $E_i$ is an exceptional curve meeting $D_i$ and $v_i$ is the primitive
generator of the ray in $\Delta_{(S',D')}$ corresponding to $D_i$. We make this construction explicit in the example in the next section. There is a canonical morphism $\nu: \Delta_{(S,D)} \rightarrow \Delta _{(S',D)}$ and in Lemmas 3.27 and 3.28 of~\cite{MirrorSymmetryforLogCY1} the authors prove that this induces an isomorphism of sheaves $\mathcal{R}(S,D) \cong \mathcal{R} (S',D)$ together with a bijection of broken lines between $\Delta_{(S,D)}$ and $\Delta _{(S',D)}$. To be precise we have the following theorem of~\cite{MirrorSymmetryforLogCY1}:

\begin{theorem}
  Let $(S,D)$ be a Looijenga pair with $(S',D)$ a toric model. Then the complex $\Delta_{S',D}$ is isomorphic as an affine manifold to $\RR^2$. Let $(\RR^2, f_{S'}, \underline{M}, \underline{M} \setminus 0)$ be the scattering diagram with values in $A_1(S)$ on $\Delta_{(S',D')}$ whose value along $\RR_{\ge 0}v_i$ is the product of incoming rays $\prod_j( 1 + z^{[E_{i,j}]+\phi(v_i)})$ and outgoing rays $\prod_j(1 + z^{-[E_{i,j}]+\phi(v_i)})^{-1}$ where the $E_{i,j}$ are the blown-down curves meeting $D_i$. Since this satisfies the assumptions of the Kontsevich-Soibelman lemma, we let $\bar{f}_{S'}$ be the associated consistent diagram.

  Let $\bar{f}_{S}$ be the canonical scattering diagram on $\Delta_{(S,D)}$. Locally $\Delta_{(S,D)}$ and $\Delta_{(S',D')}$ are isomorphic and we can evaluate $f_{S'}$, $\bar f_S$ and $\bar f_{S'}$ on the same line $\mathfrak{d}$. 
Then one has an equality
\[ \bar{f}_{S'} (\mathfrak{d}) = \bar{f}_S (\mathfrak{d}) f_{S'} (\mathfrak{d}).\]
this has the effect of replacing the base of the scattering diagram by an affine manifold isomorphic to $\RR^2$ and swapping the directions of rays $( 1 + z^{[E_{i,j}]+\phi(v_i)})$ from outgoing to incoming. This is seen explicitly in the example in the next section.
\end{theorem}

  After possibly blowing up some double points of $D$, we can then blow down appropriate curves to construct toric models, see \cite{MirrorSymmetryforLogCY1}, Proposition 1.3. This allows one to calculate the original scattering diagram by applying the Kontsevich-Soibelman lemma to $f_{S'}$. An implementation of this algorithm in Sage is available alongside this paper.

  \section{Explicit formulae}

Let us practice constructing toric models in the cases we are interested in.

\begin{example}
  Recall the dual intersection complex for $dP_5$:

  
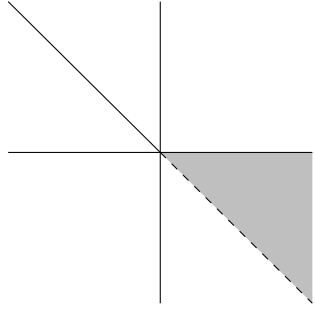
\[\begin{tikzpicture} \fill[lightgray] (0,0) -- (2,0) -- (2,-2); \draw (0,0) -- (2,0); \draw (0,0)--(0,2); \draw (0,0) -- (-2,2); \draw (0,0) -- (-2,0); \draw (0,0) -- (0,-2); \draw[dashed] (0,0) -- (2,-2); \end{tikzpicture}\]
  
  This is certainly not affine isomorphic to $\RR^2$. The boundary has five components, and this suggests that we should attempt to find a toric
model isomorphic to $dP_7$, since there is a choice of boundary divisor there also consisting of five components. A model for $dP_7$ together with its toric boundary is contained below. The red lines are a triangle in $\PP^2$, the
red circles are blown up to obtain $dP_7$, and the black points are further
blown up to obtain $dP_5$ with $dP_7$ as its toric model.
  
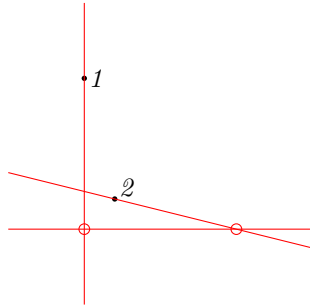
\[\begin{tikzpicture} \draw [red] (0,0) circle (2pt); \draw [red] (2,0) circle (2pt); \fill (0,2) circle (1pt); \node at (0.15,2) {1}; \fill (0.4,0.4) circle (1pt); \node at (0.55,0.55) {2}; \draw [red] (-1,0) -- (3,0); \draw [red] (0,-1)--(0,3); \draw [red] (3, -0.25) -- (-1, 0.75); \end{tikzpicture}\]
The corresponding dual complex is given by
  
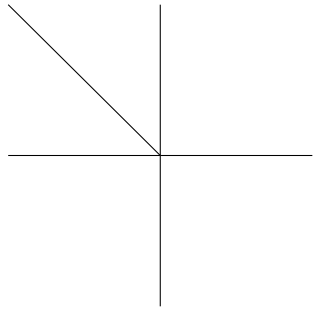
\[\begin{tikzpicture} \draw (0,0) -- (2,0); \draw (0,0)--(0,2); \draw (0,0) -- (-2,2); \draw (0,0) -- (-2,0); \draw (0,0) -- (0,-2); \end{tikzpicture}\]
In particular the self intersection of every component of the boundary for $dP_5$ was -1, whereas two of the components of the boundary of $dP_7$ have self intersection 0.

  We now incorporate this into the construction of the scattering diagram. Initially after making a choice of ordering we have have the scattering diagram:
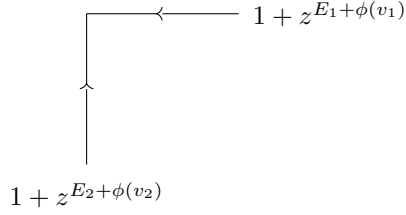
\[\begin{tikzpicture} \draw[-<] (0,0) -- (1,0); \draw[-<] (0,0) -- (0,-1); \draw (1,0)--(2,0); \draw (0,-1) -- (0,-2); \node at (3.2, 0) {$1 + z^{E_1+\phi(v_1)}$}; \node at (0, -2.4) {$1 + z^{E_2+\phi(v_2)}$};\end{tikzpicture}\]
and we saw in our motivation for the Kontsevich-Soibelman lemma that this produces a consistent diagram
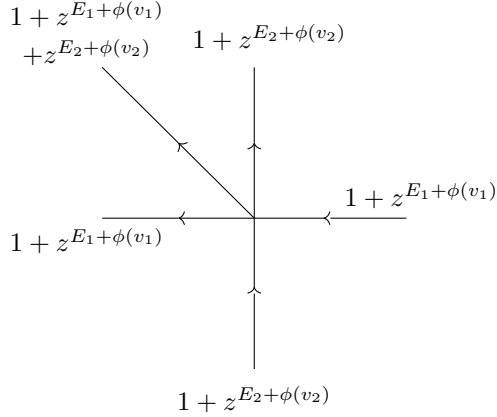
\[\begin{tikzpicture} \draw[-<] (0,0) -- (1,0); \draw[-<] (0,0) -- (0,-1); \draw (1,0)--(2,0); \draw (0,-1) -- (0,-2);\draw[->] (0,0) -- (-1,0); \draw[->] (0,0) -- (0,1); \draw (-1,0)--(-2,0); \draw (0,1) -- (0,2);\draw[->] (0,0) -- (-1,1); \draw (-1,1) -- (-2,2); \node at (2.2,0.3) {$1 + z^{E_1 + \phi(v_1)}$}; \node at (-2.2,-0.3) {$1 + z^{E_1 + \phi(v_1)}$}; \node at (-2.2,2.7) {$1 + z^{E_1 + \phi(v_1)}$}; \node at (-2.2,2.2) {$+z^{E_2 + \phi(v_2)}$}; \node at (-0,-2.4) {$1 + z^{E_2 + \phi(v_2)}$};\node at (0.2,2.4) {$1 + z^{E_2 + \phi(v_2)}$};\end{tikzpicture}\]
which agrees up to changing the direction of the incoming rays and pushing forwards with the scattering diagram found in~\cite{MirrorSymmetryforLogCY1} Figure 1.1. This is exactly as predicted by the Kontsevich-Soibelman Lemma.
\label{ex:monodromycalculation}
\end{example}

The attached implementation is optimised in three ways. Firstly one only is interested in terms of order $n$, so we may reduce all the series modulo terms of higher order and use efficient power calculations to do so. Secondly this can be distributed over multiple cores. Thirdly each ray stores a list of all powers already calculated so as to minimise repeat calculations. I would appreciate any input on how to significantly speed up this algorithm.

Since by the previous section only a finite collection of classes need be calculated we may terminate the calculation once those coefficients have been obtained and substitute them into the formulae \eqref{formulaestart}-\eqref{formulaeend} we found using the techniques in section 2.2.2.

Recall that the Picard group of $dP_d$ is generated by a hyperplane class $H$ and the class of the exceptional curves $E_1, \ldots E_{9-d}$. By using the attached implementation of the Kontsevich-Soibelman lemma one can read off the desired coefficients and the resulting families are contained in Table~\ref{explicitequations}.
 \newpage
\begin{center}
  \begin{table}
\begin{tabular}{l|ll}
  Surface & Equations & Super-potential \\
  $dP_5$ & $\vartheta_{1} \vartheta_3 = z^{[D_2]}\vartheta_2 + z^{[D_4] +[D_5]}$ & $\sum \vartheta_i$ \\
  & $\vartheta_2 \vartheta_4 = z^{[D_3]}\vartheta_3 + z^{[D_1]+[D_4]}$  & \\
  & $\vartheta_3 \vartheta_5 = z^{[D_4]}\vartheta_4 + z^{[D_1]+[D_2]}$  & \\
  & $\vartheta_4 \vartheta_1 = z^{[D_5]}\vartheta_5 + z^{[D_2]+[D_3]}$  & \\
  & $\vartheta_5 \vartheta_2 = z^{[D_1]}\vartheta_1 + z^{[D_3]+[D_4]}$  & \\
   & & \\
  $dP_4$ & $\vartheta_1 \vartheta_3 = z^{[D_2]}\vartheta_{2} + z^{[D_4]} \vartheta_4 +$  & $\sum \vartheta_i$\\
   & \quad $z^{H-E_1} + z^{2H - E_1 - E_2 - E_3 - E_5} + z^{2H - E_1 - E_2 - E_4 - E_5}$ & \\
  & $\vartheta_2 \vartheta_4 = z^{[D_1]} \vartheta_1 + z^{[D_3]} \vartheta_3 +$ & \\
  & \quad $z^{H-E_3} + z^{H-E_4} + z^{2H- E_2 - E_3 - E_4 - E_5}$ & \\
  & & \\
  $dP_3$ & $\vartheta_1 \vartheta_2 \vartheta_3 = z^{[D_1]}\vartheta_1^2 + z^{[D_2]} \vartheta_2^2 + z^{[D_3]}\vartheta_{3}^2 + $ & $\sum \vartheta_i$ \\
  & \quad $z^{[D_1]} \vartheta_1 (z^{E_1} + z^{E_2} + z^{H- E_3 - E_5} + z^{H - E_3 - E_6} + $ & \\
  & \quad $z^{H- E_4 - E_5} + z^{H- E_4 - E_6}) + z^{[D_2]} \vartheta_2 (z^{E_3} + z^{E_4} +$ & \\
  & \quad $z^{H- E_1 - E_5} + z^{H- E_1 - E_6} + z^{H- E_2 - E_5} + z^{H- E_2 - E_6})+$ & \\
  & \quad $z^{[D_3]} \vartheta_3 (z^{E_5} + z^{E_6} + z^{H- E_1 - E_3} + z^{H - E_1 - E_4}$ & \\
  & \quad $z^{H - E_2 - E_3} + z^{H - E_2 - E_4}) + z^H + $ & \\
  & \quad $(z^{2H}  + z^{4H - \sum E_i}) \alpha +$ & \\
  & \quad $z^{3H-\sum E_i}(z^{E_2-E_1} + z^{E_1-E_2} + z^{E_4-E_3} + z^{E_3-E_4} +$ & \\
  & \quad $z^{E_6-E_5} + z^{E_6-E_5})+ z^{5H-2 \sum E_i} - 4z^{[D]}$ & \\
  & $ \alpha = (z^{-E_1} + z^{-E_2})(z^{-E_3} + z^{-E_4})(z^{-E_5} + z^{-E_6})$ & \\
  & &
\end{tabular}
  \caption{Explicit equations for mirror families}
  \label{explicitequations}
\end{table}
\end{center}

The expression for the mirror family to $dP_3$ agrees with that found in~\cite{MirrorSymmetryforLogCY1}, Example 6.13. To express the formulae for $dP_2$ requires some additional notation. In this case we had to blow up two additional points which lay on the intersection of the chosen conic $C$ and line $L$. Let $E_8$, $E_9$ be the exceptional divisors lying over the intersections points of $C$ and $L$. Let us relabel the $\vartheta_i$ by the cohomology class of the corresponding \Lawrence{boundary divisor, e.g., writing $\vartheta_{D_i}$ for $\vartheta_i$}. There is an action of $C_2$ exchanging $E_1$ and $E_2$, an action of $S_5$ permuting $E_3$ through $E_7$ and an action of $C_2$ exchanging $E_8$ and $E_9$. When we write $\prod (1 +z^{h H + \sum a_i E_i})$ we mean to take the product of $ 1+z^{h[H] + \sum a_i [E_i]}$ over elements of the orbit under the action of $C_2 \times S_5$. In particular we choose here representatives with $a_1 \geq a_2$ and $a_3$ through $a_7$ decreasing. For instance $\prod (1+ z^{H+E_3})$ would be the product $(1+z^{H+E_3})(1 + z^{H+E_4})(1 + z^{H+E_5})(1 + z^{H+E_6})(1 + z^{H+E_7})$.
\Lawrence{Searching through the output of our implementation of the Kontsevich-Soibelman lemma we are able to find the coefficients we were missing in the equations \eqref{formulaestart} to \eqref{formulaeend} for the mirror family to $dP_2$ relative to our choice of boundary. After calculating these we find that the mirror family is given by}
\begin{align*}
\vartheta_C \vartheta_L &= z^{E_8} \vartheta_{E_8} + z^{E_9} \vartheta_{E_9} + 3 z^{C + L + 2E_8 + 2E_9} + B_1(E_8)z^{E_8}\\
\vartheta_{E_8} \vartheta _{E_9} &= \vartheta_{3L} z^L + \vartheta_{2L} B_1 (L) z^{L} + \vartheta_{L} B_2 (L) z^{L} \\
 & \quad + z^{C} B_3 (C) + 3 z^{2C + 2L + 3E_8 + 3E_9} + \vartheta_C B_2 (C) z^{C} \\
& \quad + \vartheta_{2C} B_1 (C) z^{C} + \vartheta_{3C} z^C\\
\vartheta_{C}^3 &= \vartheta_{3C} + 3 B_1 (L) z^{L+E_8+E_9} \vartheta_{C} + 6z^{C+2L+3E_8+3E_9} + 6 B_1 (E_9)\\
 & \quad z^{L+E_8+2E_9} + 3 z^{L + 2E_8 + E_9}\vartheta_{E_8} + 3 z^{L+ E_8 + 2E_9} \vartheta_{E_9}\\
\vartheta_L^3 &= \vartheta_{3L} + 3 B_1 (C) z^{C+E_8+E_9} \vartheta_L + 6z^{2C+L+3E_8+3E_9} + 6 B_1 (E_8)\\
 & \quad z^{C+2E_8+E_9} + 3 z^{C + E_8 + 2E_9} \vartheta_{E_9} + 3 z^{C + 2E_8 + E_9} \vartheta_{E_8}\\
\vartheta_{C}^2 &= \vartheta_{2C} + 2\vartheta_{L} z^{L+E_8+E_9} + 2B_1 (L) z^{L+E_8+E_9} \\
\vartheta_{L}^2 &= \vartheta_{2L} + 2\vartheta_{C} z^{C+E_8+E_9} + 2B_1 (C) z^{C+E_8+E_9}\\
\end{align*}
where the coefficients $B_k(E)$ are terms in the scattering diagram related to the relative invariants of curves meeting the curve $E$ in one point with tangency order $k$. The term $B_i (-)$ is the coefficient of $t^i$ in the product $\prod N_k (-,t^k)$, where the $N_k(-,t)$ are given by the following formulae:
\begin{align*}
  N_3 (C, t) &=  \prod (1+9 tz^{2H - E_1 - E_2 - E_3}) \prod ( 1+ 9tz^{3H - 2E_1 - E_2 - E_3 - E_4 - E_5} ) \prod ( 1 +9tz^{4H - 2 E_1 - 2 E_2 - 2 E_3 - E_4 - E_5 - E_6}) \\ &  \prod ( 1 +9tz^{4H - 3 E_1 - E_2 - E_3 - E_4 - E_5 - E_6 - E_7}) \prod (1 + 72tz^{4H - 2E_1 - 2 E_2 - E_3 - E_4 - E_5 - E_6 - E_7}) \\ & \prod ( 1 +9tz^{5H - 3E_1 - 2E_2 - 2E_3 - 2E_4 - E_5 - E_6 - E_7} ) \prod ( 1 +9tz^{6H - 3E_1 - 3E_2 - 2E_3 - 2E_4 - 2E_5 - 2E_6 - E_7})\\
  N_2 (C, t) &=   \prod ( 1+ 4tz^{H-E_1} ) \prod ( 1 +4tz^{2H - E_1 - E_2 - E_3 - E_4} ) \prod ( 1 +4tz^{3H - 2E_1 - E_2 - E_3 - E_4 - E_5 - E_6}) \\ & \prod ( 1 +4tz^{4H - 2E_1 - 2E_2 - 2E_3 - E_4 - E_5 - E_6 - E_7})\\
  N_1 (C, t) &= \prod ( 1+ tz^{E_3} ) \prod ( 1 +tz^{H-E_1 - E_3} ) \prod ( 1 +tz^{2H - E_1 - E_2 - E_3 - E_4 - E_5}) \\  & \prod ( 1 +tz^{3H - 2E_1 - 2E_2 - 2E_3 - E_4 - E_5 - E_6 })\\
  N_3 (L, t) &= \prod(1 + 9tz^{3H - 2E_3 - E_4 - E_5 - E_6 - E_7} ) \prod ( 1 +9tz^{4H - E_1 - 2E_3 - 2E_4 - 2E_5 - E_6 - E_7}) \\ & \prod ( 1 +9tz^{5H - E_1 - E_2 - 3 E_3 - 2 E_4 - 2E_5 - 2E_6 - E_7}) \prod ( 1 + 72tz^{5H - E_1 - E_2 - 2E_3 - 2E_4 - 2E_5 - 2E_6 - 2E_7}) \\ & \prod ( 1 +9tz^{5H - 2 E_1 - 2E_3 - 2E_4 - 2E_5 - 2E_6 - 2E_7} ) \prod ( 1 +9tz^{6H - 2E_1 - E_2 - 3E_3 - 3E_4 - 2E_5 - 2E_6 - 2E_7})\\ & \prod ( 1 +9tz^{7H - 2E_1 - 2E_2 - 3E_3 - 3E_4 - 3E_5 - 3E_6 - 2E_7})\\
  N_2 (L, t) &=  \prod (1 + 4tz^{2H - E_3 - E_4 - E_5 - E_6} ) \prod ( 1 +4tz^{3H - E_1 - 2E_3 - E_4 - E_5 -E_6 - E_7})\\ & \prod ( 1 +4tz^{4H - E_1 - E_2 - 2E_3 - 2E_4 - 2E_5 - E_6 - E_7} ) \prod ( 1 +4tz^{5H - 2E_1 - E_2 - 2E_3 - 2E_4 - 2E_5 - 2E_6 - 2E_7})\\
  N_1 (L, t) &= \prod ( 1 + tz^{E_1} ) \prod ( 1 +tz^{H - E_3 - E_4} ) \prod ( 1 +tz^{2H - E_1 - E_3 - E_4 - E_5 - E_6}) \\ & \prod ( 1 +tz^{3H - E_1 - E_2 - 2E_3 - E_4 - E_5 - E_6 - E_7})\\
\end{align*}
\begin{align*}
  N_1 (E_8, t) &= \prod (1+ tz^{H-E_3-E_8} ) \prod ( 1 +tz^{2H - E_1 - E_3 - E_4 - E_5 - E_8}) \prod(1 + 9 tz^{3H - E_1 - E_2 - E_3 - E_4 - E_5  - E_6 - E_7 - E_8}) \\ & \prod ( 1 +tz^{3 H - E_1 - E_2 - 2E_3 - E_4 - E_5 - E_6 - E_8} ) \prod ( 1 +tz^{3H - 2 E_1 - E_3 - E_4 - E_5 - E_6 - E_7 - E_8}) \\ & \prod ( 1 +tz^{4H - 2E_1 - E_2 - 2 E_3 - 2 E_4 - E_5 - E_6 - E_7 - E_8} ) \prod ( 1 +tz^{5H - 2E_1 - 2 E_2 - 2 E_3 - 2E_4 - 2E_5 - 2E_6 - E_7 - E_8})\\
  N_1 (E_9, t) &= \prod (1 + tz^{H-E_3-E_9} ) \prod ( 1 +tz^{2H - E_1 - E_3 - E_4 - E_5 - E_9}) \prod (1 + 9 tz^{3H - E_1 - E_2 - E_3 - E_4 - E_5  - E_6 - E_7 - E_9}) \\ & \prod ( 1 +tz^{3 H - E_1 - E_2 - 2E_3 - E_4 - E_5 - E_6 - E_9} ) \prod ( 1 +tz^{3H - 2 E_1 - E_3 - E_4 - E_5 - E_6 - E_7 - E_9}) \\ & \prod ( 1 +tz^{4H - 2E_1 - E_2 - 2 E_3 - 2 E_4 - E_5 - E_6 - E_7 - E_9} ) \prod ( 1 +tz^{5H - 2E_1 - 2 E_2 - 2 E_3 - 2E_4 - 2E_5 - 2E_6 - E_7 - E_9})\\
\end{align*}
After setting all of the monomials equal to one we find that the values of $B_i(-)$ reduce to $6561$, $459$, $27$, $6561$, $459$, $27$, $81$ and $81$ respectively and this fibre of the family is smooth, we will use this in calculations later. We expect that each fibre of this family is a double cover of the plane. This can be seen by projecting onto $\vartheta_C$ and $\vartheta_L$, then the equations are entirely symmetric under exchanging these two variables, noting of course that $z^C B_3(C) = z^L B_3 (L)$.

Unfortunately this is a deformation of the four-vertex, not of the two-vertex. This stems from our decision to blow up two additional points at the start, and so the family we actually want is the restriction to the locus $z^{E_8} = z^{E_9} = 1$. Setting $z^C = z^L = 0$ we see that coordinates on the central fibre are $\vartheta_C, \vartheta_L$ and $\vartheta_{E_8} - \vartheta_{E_9}$. Eliminating the extra variables we obtain an equation for the family:
\begin{align} (\vartheta_{E_8} - \vartheta_{E_9})^2 &= \vartheta_C^2\vartheta_L^2 - 4\vartheta_C^3z^C - 4\vartheta_L^3z^L + 18\vartheta_C\vartheta_Lz^{C+L} - 27z^{2C+2L} - 4\vartheta_L^2z^LB_1(L)  - 4\vartheta_C^2z^CB_1(C) \nonumber \\ & \quad + 20\vartheta_Cz^{C+L}B_1(L) + 20\vartheta_Lz^{C+L}B_1(C) + 16z^{C+L}B_1(L)B_1(C) - 4\vartheta_Lz^LB_2(L) \nonumber \\ &\quad - 4\vartheta_Cz^CB_2(C) - 2\vartheta_C\vartheta_LB_1(E_8) + 30z^{C+L}B_1(E_8) - 4z^CB_3(C) - B_1(E_8)^2 \label{dP2family} \end{align}
Now we see clearly the leading order terms $\vartheta_C^2 \vartheta_L^2 = (\vartheta_{E_8} - \vartheta_{E_9})^2$, giving rise to the two-vertex over the origin. The appropriate super-potential remains $\vartheta_C + \vartheta_L$.

One could ask if there is a better way to present the data hiding in the scattering diagram. In the cases of $dP_5$ and $dP_4$ the canonical scattering diagram is finite and can be computed, as was done in Example 3.7 of~\cite{MirrorSymmetryforLogCY1} for one of these. By an observation of Gross, Hacking and Keel the entire scattering diagram for $dP_3$ is generated by the series associated to a single ray under the birational automorphism group of this surface, generated by a finite collection of classes. This breaks down in lower degrees, as far as I am aware there is no known closed form expression for the scattering diagram associated to the final surface.

There are however more elegant ways to describe terms arising in the equation, as was seen in Example 6.13 of~\cite{MirrorSymmetryforLogCY1}. Following these ideas we can describe the terms arising in equation~\ref{dP2family}. To do this we write $N_i(-,t)^{(j)}$ for the coefficient of $t^j$ in $N_i(-,t)$. Then we find the following birational descriptions:

\begin{tabular}{ll}
  Term & Description \\
  $N_1(C,t)^{(1)}z^C$ & The sum of $\frac{1}{16}z^{\pi^*(H-E)}$ over all contractions to $dP_8$ obtained by contracting \\ & one $-1$-curve meeting $C$ and five meeting $L$, none of which meet $\pi^*E$ or $\pi^* H$. \\
  $N_1(L,t)^{(1)}z^L$ & The sum of $\frac{1}{16}z^{\pi^*(H-E)}$ over all contractions to $dP_8$ obtained by contracting \\ &  one $-1$-curve meeting $L$ and five meeting $C$, none of which meet $\pi^*E$ or $\pi^* H$. \\
  $N_1 (C,t)^{(2)} z^{C}$ & The sum of $z^{\pi^* H}$ over all contractions to $\PP^2$ obtained by blowing down\\ & two $-1$-curves meeting $C$ and five meeting $L$. \\
  $N_1 (L,t)^{(2)} z^{L}$ & The sum of $z^{\pi^* H}$ over all contractions to $\PP^2$ obtained by blowing down\\ & two $-1$-curves meeting $L$ and five meeting $C$. \\
  \end{tabular}

\begin{tabular}{ll}
  Term & Description \\
  $N_2(C,t)^{(1)}z^C$ & The sum of $4 z^{\pi^*\text{-} K_{\bar{S}}}$ over all contractions of a single $-1$-curve meeting $L$\\ & and not meeting $C$. \\
  $N_2(L,t)^{(1)}z^C$ & The sum of $4 z^{\pi^*\text{-} K_{\bar{S}}}$ over all contractions of a single $-1$-curve meeting $C$\\ & and not meeting $L$. \\
  $N_1 (E_8, t)^{(1)} z^{E_8}$ & The sum of $z^{\pi^*(H-E)}$ over all contractions to $dP_8$ obtained by contracting\\ & four $-1$-curves meeting one of $C$ or $L$ and three meeting the other.
  \end{tabular}
The term $\frac{1}{16}$ in $N_1(C,t)^{(1)}z^C$ arises because given one of these classes $\alpha$ there are sixteen $-1$-curves (each of degree at most three) meeting $\alpha$ and $C$ but not meeting $L$. Fixing any one of these produces a unique contraction contributing the desired monomial. In other words a subgroup of the birational automorphism group of the surface stabilises the class $\alpha$. Finally we observe by direct computation that $z^CN_3(C,t)^{(1)} = 9z^{C+L}N_1(E_8,t)^{(1)} - 9z^{-2K_S}$. This gives a birational interpretation of every term appearing in equation~\ref{dP2family}.

\section{Comparisons and evidence}

We begin this section by comparing the predictions of~\cite{MirrorSymmetryForDelPezzoSurfacesVanishingCyclesAndCoherentSheaves} for the mirror families. These predictions state the mirror to $dP_n$ ought to be a compactification of a fibration of the two-torus $W:\mathbb{G}_m^2 \rightarrow \CC$ with an $I_{n}$ fibre at infinity. Here $W$ is the super-potential. For example,
in the case of $dP_5$ we can take equations:
\[x_{i+1} x_{i-1} = x_i + 1\]
defining a surface inside $\AA^5$ with $W=\sum x_i$. It turns out this
is already a partial compactification of the mirror to $dP_5$.
To further compactify, we take the projective closure of the above
surface, taking coordinates $Z_1,\ldots,Z_5$ and $T$ to obtain equations
\[
Z_{i+1}Z_{i-1}=Z_i T + T^2.
\]
Then $T$, $W$ can be viewed as sections of ${\mathcal O}(1)$ on this
surface, hence giving a pencil. After resolving the base-points of this
pencil, one obtains the desired compactification. 
 The fibre at infinity therefore has equations
\[Z_{i-1} Z_{i+1} = 0\]
 for $i \in \{1, 2, 3,4,5\}$  inside $\PP^4$, hence defining a cycle of five planes. This is exactly as predicted by~\cite{MirrorSymmetryForDelPezzoSurfacesVanishingCyclesAndCoherentSheaves}. For the other del Pezzo surfaces
$dP_n$ with $n\ge 4$, the same shape of the mirror as defined by
quadratic equations confirms that the fibre at infinity will be an $I_n$ fibre.
More careful considerations with the equations show the same for $n=3$ and $2$.

There are two other tests available to us to determine whether the families described above are indeed the mirror families. The first we implement here, whilst the second is of wider interest.

\subsection{Dimension of the Jacobian ring}

The rank of the quantum cohomology is equal to the rank of the regular cohomology, and therefore we already know the expected rank of all of the Jacobian rings in each case we have constructed. If $S$ is a del Pezzo surface of degree $d$ then the rank of the middle homology is $10-d$, generated by a hyperplane and the exceptional curves. Therefore the rank of the quantum cohomology should be $12-d$. Now let $(\check{V}, \check{W})$ be a smooth fibre of the family such that $\check{W}$ has isolated critical points. The rank of the Jacobian ring of $\check{W}$ should then be equal to $12-d$. This is the weakest test since it tells us no enumerative information about the surface. This rank is at least sensitive to deforming the equations defining the mirror.

Let us attempt this calculation for the case of the mirror to $dP_2$. First recall that the super-potential on Landau-Ginzburg model is the partial sum $x_C + x_L$. If one took the full sum one would construct the mirror to the non-Fano double blow up of $dP_2$, and we do not believe that this mirror construction should apply.

\begin{example}

  Recall that the Jacobian ring of a potential $W: \CC^2 \rightarrow \CC$ is the quotient \[ \CC [u_1, u_2] / \langle \partial \check{W} / \partial u_i \rangle \]
  We are not quite in this situation and so we need to be more careful. The Jacobian ring of $(S,W: S \rightarrow \CC)$ is the function ring of the critical points of $W$. In our case $S$ is defined by two equations in $\AA^4$, $f_1$ and $f_2$. Therefore we are interested in the points of $S$ where $\nabla W$ is contained in the span of $\nabla f_1$ and $\nabla f_2$. This asks that all the $3 \times 3$ minors of the following matrix vanish

  \[
M=
  \begin{bmatrix}
    \partial f_1 / \partial \vartheta_C & \partial f_1 / \partial \vartheta_L & \partial f_1 / \partial \vartheta_8 & \partial f_1 / \partial \vartheta_9 \\
    \partial f_2 / \partial \vartheta_C & \partial f_2 / \partial \vartheta_L & \partial f_2 / \partial \vartheta_8 & \partial f_2 / \partial \vartheta_9 \\
    \partial W / \partial \vartheta_C & \partial W / \partial \vartheta_L & \partial W / \partial \vartheta_8 & \partial W / \partial \vartheta_9 \\
    
  \end{bmatrix}
  \]
  
  This can be calculated using Sage and generating elements of the defining ideal are:
  \begin{multline}(\vartheta_9 + 2)\vartheta_C + 3\vartheta_C^2 - (\vartheta_9 + 2)\vartheta_L - 3\vartheta_L^2 + 58\vartheta_C - 58\vartheta_L, (\vartheta_8 + 2)\vartheta_C + 3\vartheta_C^2 - (\vartheta_8 + 2)\vartheta_L - 3\vartheta_L^2 + 58\vartheta_C - 58\vartheta_L, \\ -\vartheta_8 + \vartheta_9, -\vartheta_8 + \vartheta_9, \vartheta_C\vartheta_L - \vartheta_8 - \vartheta_9 - 84, \\ -\vartheta_C^3 - \vartheta_L^3 + \vartheta_8\vartheta_9 - 27\vartheta_C^2 + 4\vartheta_C\vartheta_L - 27\vartheta_L^2 + 2\vartheta_8 + 2\vartheta_9 - 324\vartheta_C - 324\vartheta_L - 3000 \end{multline}
  Sage allows us to calculate a $k$-basis of the quotient ring and we find that the quotient is 10-dimensional as a vector space, exactly as hoped for.
\end{example}

\subsection{Structure of the Jacobian ring}

In favourable situations we may calculate both structures, the quantum cohomology and the Jacobian ring. The quantum cohomology for high degree del Pezzo surfaces was calculated in~\cite{QuantumCohomologyOfRationalSurfaces}. In section 4 they produce a list of relevant curve classes and in section 5 describe the product as a sum over those classes.

The question is then what the conjectured isomorphism is. For a divisor $D$ it was suggested in~\cite{LagrangianIntersectionFloerTheoryAnomalyAndObstruction} that this map should send $D$ to the derivative of $W$ with respect to the log vector field $z^{D} \partial z^D$. This of course leaves the question of where a point should be sent. This can be computed by computing products on both sides, for instance once we know where a hyperplane is sent we may compute its square to find the class of a point, subtracting any quantum corrections.

To do this calculation explicitly would be un-enlightening. Rather it would be an interesting project to attempt to automate this construction. Sage has the necessary computational tools to carry this out. In future work I hope to perform the same analysis of the quantum cohomology as was done in~\cite{QuantumCohomologyOfRationalSurfaces}, but for the degree two del Pezzo surface. The already know examples would provide a fertile collection of examples with which to demonstrate the algorithm, whilst simultaneously allowing us to demonstrate that our construction here is correct.

\bibliographystyle{plain}
\bibliography{Draft}

\quad \newline
\noindent
\texttt{National Center for Theoretical Sciences\\ No. 1 Sec. 4
Roosevelt Rd., National Taiwan University\\ Taipei, 106, Taiwan}

\end{document}